\documentclass[12pt, reqno]{amsart}
\usepackage{amsmath, amsthm, amscd, amsfonts, amssymb, graphicx, color, mathrsfs,enumitem}
\usepackage[bookmarksnumbered, colorlinks, plainpages]{hyperref}
\usepackage[all]{xy}
\usepackage{slashed}
\usepackage{mathabx}
\usepackage{tipa}
\usepackage{soul}
\usepackage{cancel}
\usepackage{ulem}
\usepackage{mathtools}
\newcommand{\eN}{\mathbb{N}}
\def\gy{{(1+G_{Y}(X-Y))}}

\newtheorem{theorem}{Theorem}[section]
\newtheorem{lemma}[theorem]{Lemma}

\newtheorem{proposition}[theorem]{Proposition}
\newtheorem{corollary}[theorem]{Corollary}
\theoremstyle{definition}
\newtheorem{definition}[theorem]{Definition}

\theoremstyle{remark}
\newtheorem{remark}[theorem]{Remark}  
\numberwithin{equation}{section}

\newcommand\norm[1]{\left\lVert#1\right\rVert}
\newcommand{\jpx}{\langle x \rangle}
\newcommand{\subscript}[2]{$#1 _ #2$}
\newcommand{\jpxi}{\langle \xi \rangle}
\newcommand{\jp}{\langle \xi \rangle}
\newcommand{\jpxe}{\langle X \rangle}
\newcommand{\jpxet}{\langle (x, \xi) \rangle}
\newcommand{\jpX}{\langle X \rangle}
\newcommand{\jpe}{\langle \eta \rangle}
\newcommand{\jpye}{\langle Y \rangle}
\newcommand{\half}{\frac{1}{2}}
\newcommand{\Ra}{\mathbb{R}}
\newcommand{\ene}{\mathbb{N}}
\newcommand{\epsi}{\varepsilon}
\newcommand{\Ran}{\mathbb{R}^n}
\newcommand{\frn}{\frac{n}{2}}
\DeclareMathOperator{\os}{o}

\begin{document}
\setcounter{page}{1}

\title[Degenerate Schr\"odinger equations]{Degenerate Schr\"odinger equations with irregular potentials}

\author[D. Cardona]{Duv\'an Cardona}
\address{
  Duv\'an Cardona:
  \endgraf
  Department of Mathematics: Analysis, Logic and Discrete Mathematics
  \endgraf
  Ghent University, Belgium
  \endgraf
  {\it E-mail address} {\rm duvan.cardonasanchez@ugent.be}
  }
  
  \author[M. Chatzakou]{Marianna Chatzakou}
\address{
  Marianna Chatzakou:
  \endgraf
  Department of Mathematics: Analysis, Logic and Discrete Mathematics
  \endgraf
  Ghent University, Belgium
  \endgraf
  {\it E-mail address} {\rm Marianna.Chatzakou@UGent.be}
  }
  
  \author[J. Delgado]{Julio Delgado}
\address{
  Julio Delgado:
  \endgraf
  Departmento de Matematicas
  \endgraf
  Universidad del Valle
  \endgraf
  Cali-Colombia
  \endgraf
    {\it E-mail address} {\rm delgado.julio@correounivalle.edu.co}}

\author[M. Ruzhansky]{Michael Ruzhansky}
\address{
  Michael Ruzhansky:
  \endgraf
  Department of Mathematics: Analysis, Logic and Discrete Mathematics
  \endgraf
  Ghent University, Belgium
  \endgraf
 and
  \endgraf
  School of Mathematical Sciences
  \endgraf
  Queen Mary University of London
  \endgraf
  United Kingdom
  \endgraf
  {\it E-mail address} {\rm michael.ruzhansky@ugent.be}
  }



\thanks{The authors are supported by the FWO Odysseus 1 grant G.0H94.18N:Analysis and Partial Differential Equations and by the Methusalem programme of the Ghent University Special Research Fund(BOF)(Grant number 01M01021). J. Delgado is also supported by Vice. Inv. Universidad del Valle Grant CI 71329, MathAmSud and Minciencias-Colombia under the project MATHAMSUD21-MATH-03. Marianna Chatzakou is also supported by the FWO Fellowship grant No 12B1223N. Michael Ruzhansky is also supported by EPSRC grant EP/R003025/2.}

\begin{abstract} In this work we investigate  a class  of degenerate  Schr\"odinger equations  associated to degenerate elliptic operators with irregular potentials on $\Ran$   by introducing a suitable H\"ormander metric $g$ and a  $g$-weight $m$. We establish the  well-posedness for the corresponding  degenerate Schr\"odinger and degenerate parabolic equations. When the subelliticity is available on the degenerate elliptic operator we  deduce spectral properties for a class of degenerate Hamiltonians. We also study the $L^p$ mapping properties for operators with symbols in the $S(m^{-\beta},g)$ classes in the spirit of classical Fefferman's $L^p$-bounds for the $(\rho, \delta)$ calculus. Finally, within our $S(m,g)$-classes,  sharp $L^p$-estimates and Schatten properties for Schr\"odinger operators for H\"ormander sums of squares are also investigated.
\end{abstract} \maketitle

\tableofcontents
\allowdisplaybreaks
\section{Introduction}

\subsection{Outline}
In this work we study a class of degenerate Schr\"odinger equations corresponding to a  Hamiltonian:  $\mathcal{H}_V=a_2(x,D)+V(x)$, where $a_2(x,D)$ is a second order degenerate elliptic operator on $\Ran$ and  the potential $V$ is a real-valued measurable function of quadratic order at  $\infty$.  Our analysis covers the case where the principal symbol $a_2(x,\xi)$ of the operator $a_{2}(x,D)$ which we assume to be positive belongs to the Kohn-Nirenberg class of second order, that is, it satisfies estimates of the type
\begin{equation}\label{KN:second:order}
    |\partial_{x}^{\beta}\partial_\xi^\alpha a_{2}(x,\xi)|\leq C_{\alpha,\beta}(1+|\xi|)^{2-|\alpha|},\,\,(x,\xi)\in T^*\mathbb{R}^n.
\end{equation}Note that by the spectral theory of second-order self-adjoint operators, the positivity condition $a_{2}(x,\xi)\geq 0$ also covers the case of  differential operators given by
\begin{equation}\label{oprelss} a_2(x,D)=-\sum a_{ij}(x)\frac{\partial^2}{\partial x_{i} \partial x_{j}}f+ \mbox{ lower order terms}, \quad\quad f\in C_{0}^\infty (\Ran),\end{equation}
with the coefficients $a_{ij}\in C_b^\infty (\Ran)$, belonging to  the space of real-valued smooth functions which are uniformly bounded on $\Ran$, together with all their derivatives, and so that  $A(x)=(a_{ij}(x))$ is a positive semi-definite matrix for every $x\in\Ran .$ Then, our setting also includes H\"ormander sub-Laplacians 
$$ a_2(x,D)=-\sum\limits_{j=1}^k X_j^2$$ that are sums of squares of real vector fields.  For a differential operator $a_{2}(x,D)$ whose symbol is as in \eqref{KN:second:order}, we introduce a metric $g$ on the phase-space defined by

 \begin{equation}\label{metiya}
    g_{X}(dx,d\xi):=\frac{(\langle \xi \rangle^{2}+|x|^2)}{m(x,\xi)}\left(dx^{2}+\frac{d\xi^{2}}{\langle \xi \rangle^{2}+|x|^2}\right)\,,
\end{equation}
where \begin{equation}\label{weight}m(x,\xi)=a(x,\xi)+ \jpxet, \end{equation}
with
\[a(x,\xi)=a_2(x,\xi)+|x|^2,\,\,X=(x,\xi),\,\, \jpxet=(1+|x|^2+|\xi|^2)^{\half}.\]%
The metric and the weight above are sharp with respect to the H\"ormander conditions: temperateness, slow continuity, and the uncertainty principle, that defines the $S(m,g)$ class that contains the Schr\"odinger operators $a_{2}(x,D)+|x|^2.$

After establishing  that we are in a suitable $S(m,g)$ pseudodifferential calculus, i.e, the metric $g$ is indeed a H\"ormander metric and $m$ is a $g$-weight. We consider  $V$ to be a real-valued Borel function defined on $\Ran$  such that for appropriate constants $C, C_1,  C_2>0$,  $|V(x)|\leq C|x|^2,$ for a.e.$|x|\geq C_1$ and $V(x)\geq -C_2$, for a.e.  $x\in\Ran$. Then, by using the machinery of the Weyl-H\"ormander calculus, we:
\begin{itemize}
    \item obtain the well-posedness of the following Cauchy problem on the Sobolev space $H(M,g)$:
\begin{equation}\label{EQ:cp2xertaax:2:intro:2}
\left\{\begin{array}{rl}
i\partial_t u& = \mathcal{H}_Vu, \\
u(0)&= f.
\end{array}
\right.
\end{equation}
\item We analyse the well posedness for degenerate parabolic equations of the form 
\begin{equation}\label{EQ:cpxaax1ff1he:2:intro}
\left\{\begin{array}{rl}
\partial_t u & = \mathcal{H}u, \\
u(0)&= f,
\end{array}
\right.
\end{equation} on  $L^2(\mathbb{R}^n).$ The Hamiltonian $\mathcal{H}$ is considered to be irregular.

\item We establish the $L^p$-boundedness theory for the $S(m^{-\frac{n\varepsilon}{2}},g)$-classes with the metric $g$ and the weight $m$ as in \eqref{metiya} and in \eqref{weight}, respectively. Moreover, we give the sharp index $\varepsilon$ in the case where $a_{2}(x,D)$ is the principal symbol of a H\"ormander sub-Laplacian. Consequently, we obtain the $L^p$-boundedness of negative powers for our degenerate  Schr\"odinger type operators in that case.

\item We investigate the membership of  the operators in the classes $S(m^{-\mu},g)$ to the Schatten-von Neumann classes $S_{r}(L^2(\mathbb{R}^n)).$ Consequently, the information about the distribution of the eigenvalues for a class of  Schr\"odinger type operators is obtained, as well as the rate of growth of their eigenvalues.
\end{itemize}

\subsection{State of the art and contributions}

The study of the interplay of a Hamiltonian between Schr\"odinger semigroups and heat semigroups in order to analyse the behaviour of the eigenfunctions of the Hamiltonian on $L^p$ spaces has been of an active interest in the last decades since the seminal works of B. Simon and Carmona  \cite{BSimon1}, \cite{BSimon2},  \cite{BSimon3}, \cite{Carmona1}. See also \cite{albev:hs}, \cite{iou:a1s}, \cite{ishik:a1} for further  developments on this matter. Thus, the presentation here for the $L^2$ setting is somewhat in the same spirit. We will also obtain spectral properties for Hamiltonians of the form $\mathcal{H}_V+J_{\jpxet}$, where $J_{\jpxet}$ is the 
 pseudodifferential operator with symbol $\jpxet$, by means of the analysis of Schatten-von Neumann classes and where we will require some growth condition on the potential $V$ at $\infty$ to guarantee the existence of  a discrete spectrum, as it is well-known from the classical elliptic theory of Schr\"odinger operators. This will be the prototypical case if e.g. $V=|x|^2.$  We obtain the spectral properties for $\mathcal{H}_V+J_{\jpxet}$ by studying  negative powers of such operators in suitable  Schatten-von Neumann classes. 
 
In order to avoid some eventual misunderstanding regarding the terminology
of degeneracies in this work and related works on the Schr\"odinger equation we
recall some points. In the context of the study of the Schr\"odinger equation, it is
customary to use the term degenerate in a different sense. Indeed, in quantum
mechanics, an energy level $E$ is called degenerate if there are two or more states
corresponding to the energy level $E$ of the quantum system. The number of different states corresponding to $E$ is known as the degree of degeneracy of the level in that setting. The investigation of the hypoellipticity and $L^2$ estimates for degenerate Schr\"odinger operators has been considered by Y. Morimoto in \cite{mo:m1} as an extension of the Theorem 4 in Chap. II of the seminal work of C. Fefferman \cite{fe:unc}. 

An important feature of our results up to this point is that we do not impose any subellipticity condition on the operator $a_2(x,D)$, allowing to have for instance sum of squares without the H\"ormander condition, and in particular infinitely degenerate elliptic operators. It is also clear that our metric neither satisfies the so-called strong uncertainty principle (\cite{Schrohe1x}, \cite{NicolaRodino}, \cite{NicolaRodino2}), nor  the Beals condition (\cite{Be}, \cite{Be2}) to guarantee the  $L^p$ boundedness of operators with simbols in $S(1,g)$, making it somehow  patological from other perspectives.

Moreover, we also establish the well-posedness for degenerate parabolic equations when $$ a_2(x,D)=-\sum\limits_{j=1}^k X_j^2$$  is a sum of squares of real vector fields. We also extend the above results on the well-posedness for positive powers of the Hamiltonian. Further, by using the so-called {\it geodesic temperateness} of the metric $g$ (\cite{le:book}, \cite{Bonygtemp}) and assuming that  $a_2(x,D)$ is subelliptic, we prove that
\begin{equation}\label{hami1}
    \mathcal{H}_V+C:H(m,g)\rightarrow L^2(\Ran),
\end{equation}
is an isomorphism for $C>0$ large enough.

In particular, one can consider the important case of  $a_2(x,\xi)$ being the symbol of a  sum of squares \[L=-\sum\limits_{j=1}^kX_j^2,\] where $X_j$ are real vector fields on $\Ran$.  In the case of $n=2$, we will give an improved version of the $L^p$ bounds due to the specific knowledge of the H\"ormander condition when  $a_2(x,\xi)=\xi_{1}^{2}+\tilde{x_1}^{2}\xi_{2}^{2}$, where $\tilde{x_1}$ is given by \eqref{x}. The (rescaled) operator under investigation that is now expressed as follows
\begin{equation}\label{daho}
    a(x,D):=-(\partial_{x_{1}}^{2}+\tilde{x}_{1}^{2}\partial_{x_{2}}^{2})+|x|^{2}\,,
\end{equation}
boils  down to a particular case of the so-called \textit{degenerate Schr\"odinger operator}, also known as a \textit{degenerate harmonic oscillator}. To justify the use of term ``degenerate harmonic oscillator'' let us consider the case where $\tilde{x}_{1}\equiv 1$; in this case the second order differential operator on $\mathbb{R}^2$ given by $$a(x,D)=-(\partial_{x_{1}}^{2}+\partial_{x_{2}}^{2})+|x|^{2}$$ corresponds to the quantum harmonic oscillator on $\mathbb{R}^2$; see e.g. the monograph \cite{P10}. In the degenerate situation described by \eqref{daho} the operators are parametrised by a suitable family of smooth functions, $\tilde{x}_{i}\in C^{\infty}(\mathbb{R})$ that satisfy 
\begin{align}\label{x}
    \tilde{x}_{1}:= \left\{
        \begin{array}{ll}
            \textnormal{sgn}(x_1)\,x_{1}, & \quad |x_1| \leq 2 \\
            c'\,\textnormal{sgn}(x_1),& \quad |x_1| \geq 4.
        \end{array}
    \right.
\end{align}where $c'\in \mathbb{R},$ with $c'\neq 0
$.

Our next aim is the mathematical investigation of the qualitative $L^p$ mapping properties for the negative powers of our degenerate Schr\"odinger operators within the  $S(m^{-t},g)$ classes. For other works on different directions on  the  analysis   of degenerate  Schr\"odinger equations; see for example, Doi \cite{Doi94,Doi96}, Hara \cite{Hara92}, Ichinose \cite{Ichinose84,Ichinose87}, Hiroshi and Kajitani \cite{HiroshiKajitani}, Mizohata \cite{Mizohata85} and Takeuchi \cite{Takeuchi} and the references therein.\\

The $L^\infty$-$BMO,$ the $H^1$-$L^1$ and the $L^p$-boundedness for the pseudo-differential operators in the classes  $S(m^{-t},g)$ will be studied in Section \ref{boundednnesssection}, generalising to our context the estimates in Fefferman \cite{fe}. It is important to point out that in the  study of $L^p$ bounds for  pseudo-differential operators in $S(m,g)$ classes one has to take into account Beal's condition. Indeed, R. Beals has characterised in \cite{Be} and \cite{Be2} the H{\"o}rmander's metrics $g$ ensuring 
the $L^p$ boundedness for operators in $OpS(1,g)$ and $1<p<\infty$. The condition discovered by Beals for a split metric $g_{x,\xi}(z,\theta)=g_{x,\xi}(z,0)+g_{x,\xi}(0,\theta) $ is  
\[ g_{x,\xi}(0,\theta)\leq C .\]%
 In the case of $S_{\rho ,\delta}^{0} $ classes, the Beals condition is nothing else but the well known restriction $\rho=1$. For our metric $g$ this condition is equivalent to the inequality %
\[g_{\Tilde{x},\Tilde{\xi}}(0,\xi)=\frac{|\xi|^2}{a(x,\xi)+\jpxe}\leq C,\]
which does not hold unless $a$ is elliptic. Since we are  interested in the degenerate case, we do not dispose of the $L^p$ boundedness for operators in $OpS(1,g)$, however one can obtain $L^p$ boundedness for certain intervals centered at $p=2$ and suitable negative powers of the weight $m(x,\xi)$ defined by \eqref{weight}. The latter is a well known situation in the case of $S_{\rho ,\delta}^{-t} $ classes when $\rho <1$ and has been established by Fefferman (\cite{fe}). Other subsequent generalizations for $S(m,g)$-classes of the Fefferman boundedness theorem \cite{fe} have been established by one of the authors of this work in \cite{Delgado2006,Profe2}.  Extensions of the Fefferman results to other non-commutative structures can be found in \cite{CRD19} and \cite{DR19}. $L^p$-bounds in the context of the harmonic oscillator were investigated in \cite{CR}.

Regarding the boundedness properties for the class $\textnormal{Op}(S(m^{-t},g)),$ we found the following facts. For a non-negative symbol $a_2=a_2(x,\xi)\in S^2(\Ran\times\Ran)$  of a differential operator, we have:
\begin{itemize}
    \item[$\bullet$] If $\frac{3}{4}\leq \beta<1,$ then $\textnormal{Op}(S(m^{-\frn\beta},g))\subset \mathscr{B}(L^\infty(\mathbb{R}^n), BMO(\mathbb{R}^n)).$
    \vspace{0.4cm}
    \item[$\bullet$] If   $0\leq \beta<\frac{3}{4}$ and  $\sigma(x,D)\in \textnormal{Op}(S(m^{-\frn\beta},g))$, then $\sigma(x,D):L^p(\mathbb{R}^n)\rightarrow L^p(\mathbb{R}^n)$ extends to a bounded operator, provided that
\begin{equation}
    \left| \frac{1}{p}-\frac{1}{2}\right|\leq \frn\beta.
\end{equation}
\end{itemize}%

This paper is organized as follows. In Section \ref{preliminaries} we will present the necessary background on the theory of pseudo-differential operators that we will use in this paper, in particular some basics about the Weyl-H\"ormander classes.  In Section \ref{Metric}, we prove, among other things, that the metric $g$ defined in \eqref{metiya} is a H\"ormander metric and the weight $m$ is indeed, a $g$-weight. In Section \ref{applicatione} we present some applications to the well-posedness for degenerate  Schr\"odinger operators and degenerate parabolic equations as was summarised above in the $L^2$-setting.  Sections \ref{boundednnesssection} and \ref{improvedsec} are devoted to the study of the $L^p$ boundedness for the classes $S(m^{-\beta},g)$ as we have  described above. We finish the paper with Section \ref{Schattensec} where we deduce some consequences on spectral properties for our degenerate Schr\"odinger operators.

\section{Preliminaries: Weyl-H\"ormander calculus}\label{preliminaries}
In this section we summarise some preliminary topics about the Weyl quantisation and the Weyl-H\"ormander calculus. For this we will follow  H\"ormander \cite[Section 18.5]{HormanderBook34} and Chapter 2 of \cite{le:book}.\\

The Weyl quantisation of $a\in \mathscr{S}'(\mathbb{R}^n\times \mathbb{R}^n)$ is the operator $a(x,D):\mathscr{S}(\mathbb{R}^n)\rightarrow \mathscr{S}'(\mathbb{R}^n)$ defined in the weak sense by the symmetric expression
\begin{eqnarray*}
    a^{w}(x,D)u(x):=\int\limits_{\mathbb{R}^n}\int\limits_{\mathbb{R}^n}e^{2\pi i\langle x-y,\xi \rangle}a((x+y)/2,\xi)u(y)dyd\xi,\,\,u\in \mathscr{S}(\mathbb{R}^n).
\end{eqnarray*}
The Weyl quantisation of $a(\cdot,\cdot)$ is the special case of the $\tau$-quantisation of $a(\cdot,\cdot)$ defined for $0\leq \tau\leq 1,$ by
\begin{equation*}
    a^{\tau}(x,D)f(x):=\int\limits_{\mathbb{R}^n}\int\limits_{\mathbb{R}^n}e^{2\pi i\langle x-y,\xi \rangle}a(\tau x+(1-\tau)y,\xi)u(y)dyd\xi,
\end{equation*}also defined in the weak sense. So, the Weyl quantisation of $a(\cdot,\cdot)$ is  $(1/2)$-quantisation of $a(\cdot,\cdot),$ i.e. $a^{\frac{1}{2}}(x,D)=a^{w}(x,D).$ The $1$-quantisation (Kohn-Nirenberg quantisation) of $a(\cdot,\cdot),$ recovers the usual definition of the pseudo-differential operator associated to the `symbol'  $a(\cdot,\cdot):$
\begin{equation*}
    a(x,D)u(x)\equiv a^1(x,D)u(x)=\int\limits_{\mathbb{R}^n}e^{2\pi i\langle x,\xi \rangle}a(x,\xi)\widehat{u}(\xi)d\xi,
\end{equation*}where the equivalence of both representations is a consequence of the fact that the Fourier transform is an isomorphism on the Schwartz space implying that $\widehat{u}\in \mathscr{S}(\mathbb{R}^n)$ where by abuse of notation we have written $a(x,\xi)\widehat{u}:=a(x,\xi)\widehat{u}(\xi)$ which makes sense if e.g. the distribution $a(\cdot,\cdot)$ agrees with a locally integrable function. For $a,b\in \mathscr{S}(\mathbb{R}^n),$ and $0\leq \tau,\tau'\leq 1,$ we have the equality $a^{\tau}(x,D)=b^{\tau'}(x,D)$ if and only if
\begin{equation*}
    a(x,\xi)=\int\limits_{\mathbb{R}^n}e^{-2\pi i\langle \xi-\eta,z \rangle}b(x+(\tau'-\tau)z,\eta)d\eta,
\end{equation*}which means that we always can pass from one quantisation to other one by using the semigroup $(J_t)_{t\in \mathbb{R}},$ defined by
\begin{equation*}
J_t a(x,\xi):= \int\limits_{\mathbb{R}^n}e^{-2\pi i\langle \xi-\eta,z \rangle}b(x+tz,\eta)dyd\eta,   
\end{equation*}which satisfies $J_{t+s}=J_{t}J_{s},$ $t,s\in \mathbb{R}.$ In particular, for the Weyl quantisation and the Kohn-Nirenberg quantisation of $a$, respectively, we have
\begin{equation*}
    a(x,D)=(J_{1/2}a)^{w}(x,D),\quad \, a^{w}(x,D)=(J_{-1/2}a)(x,D).
\end{equation*} 
\begin{remark}[The adjoint of the Weyl quantisation]
The following remarkable property: $a^{w}(x,D)^{*}=\overline{a}^{w}(x,D),$ implies that $a^{w}(x,D)^{*}={a}^{w}(x,D)$ for real valued symbols and  it is one of the reasons why Hermann Weyl introduced this quantisation for the purposes of quantum mechanics (see H\"ormander \cite[page 151]{HormanderBook34}). The other important property is its symplectic invariance that we will not discuss here. 
\end{remark} 
\begin{remark}[Composition in the Weyl quantisation]
Let us write $X=(x,\xi),$ (or $Y=(y,\eta)$ etc.) for points on the phase space $\mathbb{R}^{2n}=\mathbb{R}^{n}_{x}\times \mathbb{R}^{n}_{\xi},$ (resp. $\mathbb{R}^{n}_{y}\times \mathbb{R}^{n}_{\eta}).$
The composition $a^{w}\circ b^{w}(x,D)\equiv a^{w}(x,D)\circ b^{w}(x,D)$ in the Weyl quantisation is related with the symplectic structure of $\mathbb{R}^n\times\mathbb{R}^n=T^* \mathbb{R}^n$. Indeed, for
 $a,b\in S(\mathbb{R}^n)$, let us define $$(a\#b)(X):=\frac{1}{\pi^{{2n}}}\int\limits_{\mathbb{R}^{2n} }\int\limits_{\mathbb{R}^{2n} }e^{-2i\sigma(X-Y_{1},X-Y_{2})}a(Y_{1})b(Y_{2})dY_{1}dY_{2},$$
where $\sigma(X,Y):=\langle y,\xi\rangle-\langle x,\eta\rangle$  is the symplectic form on $\mathbb{R}^{2n}.$ Then, in terms of the operation $\#,$  we  have $a^w(x,D)\circ b^{\omega}(x,D)=(a\#b)^{\omega}(x,D)$.
\end{remark}
In order to introduce the Weyl-H\"ormander calculus we need the notion of a H\"ormander metric. 
One reason for this, is that if we want to study the properties of some operator $L,$ we can to associate to $L$ this kind of metric.
\begin{definition}[H\"ormander's metric]\label{HM}
For $X \in \mathbb{R}^{2n}$, let $g_{X}(\cdot)$ be a positive definite quadratic form on $\mathbb{R}^{2n}$. We say that 
$g(\cdot)$ is a H\"ormander metric if the following three conditions are satisfied:
\begin{enumerate}
\item[i.] {\bf Continuity or slowness.} There exist a constant $C>0$ such that 
$$ g_{X}(X-Y)\leq C^{-1}\Longrightarrow \left(\frac{g_X(T)}{g_Y(T)}\right)^{\pm 1}\leq C,$$
for all $X,Y \in\mathbb{R}^{2n}$, $T\in\mathbb{R}^{2n}\setminus\{0\}$.
\item[ii.] {\bf Uncertainty principle.} In terms of the symplectic form $\sigma(Y,Z):=\langle z,  \eta\rangle -\langle y, \zeta\rangle$,
we define \begin{equation}\label{gsigma}
    g_{X}^{\sigma}(T):=\sup_{W\neq 0}
\frac{\sigma(T,W)^{2}}{g_{X}(W)}.
\end{equation}   We say that $g$ satisfies the {\it uncertainty principle }
if $$ \lambda_{g}(X)=\inf_{T\neq 0}
\left(\frac{g_{X}^{\sigma}(T)}{g_{X}(T)}\right)^{1/2}\geq 1 ,$$ for
all $X, T\in\mathbb{R}^{2n}$.
 
\item[iii.] {\bf Temperateness.} We say  that $g$ is temperate if there exist $\overline{C}>0$ and $J\in \mathbb{N}$ such that
$$ \left( \frac{g_{X}(T)}{g_{Y}(T)}\right)^{\pm1}\leq \overline{C}(1+g_{Y}^{\sigma}(X-Y))^{J},
$$ for all $X,Y,T\in \mathbb{R}^{2n}$.
\end{enumerate}
\end{definition}

Let $g$ be a H{\"o}rmander metric.  {\it {The uncertainty parameter}} or {\it the Planck function} associated to $g$ is defined by
  \begin{equation*}\label{hgn1}
       h_g(X)^2=\sup _{T\neq 0} \frac{g_{X}(T)}{g_{X}^{\sigma}(T) }.
  \end{equation*} 
Note that since $h_g(X)=(\lambda_{g}(X))^{-1}$, the uncertainty principle can be translated into the condition
$$ h_g(X)\leq 1 .$$

\begin{definition}[g-weight]\label{GW}  Let $M:\mathbb{R}^{2n}\rightarrow \mathbb{R}^{+}$ be a function. Then,
\begin{itemize}
\item we say that 
$M$ is $g$-{{\it continuous}} if there exists $\tilde{C}>0$
such that
$$ g_{X}(X-Y)\leq \frac{1}{\tilde{C}}\Longrightarrow\left( \frac{M(X)}{M(Y)}\right)^{\pm1}\leq \tilde{C}\,.$$
for all $X, Y\in\mathbb{R}^{2n}$.
\item we say that
$M$ is  $g$-{\it temperate} if there exist $\tilde{C}>0$ and $N\in \mathbb{N}$
such that
$$ \left( \frac{M(X)}{M(Y)}\right)^{\pm1}\leq \tilde{C}(1+g_{Y}^{\sigma}(X-Y))^{N}\,,$$
for all $X, Y\in\mathbb{R}^{2n}$.
\end{itemize}
We will say that  $M$ is a $g$-{\it weight}  if it is $g$-continuous and $g$-temperate.
\end{definition}
We can now define the classes of symbols adapted to a H\"ormander metric $g$ and a $g$-weight $M$.   
\begin{definition} For a H{\"o}rmander metric $g$ and a $g$-weight $M$, the class $S(M,g)$ consists of all smooth functions $\sigma\in C^\infty(\mathbb{R}^{2n})$ such that for any $k \in \mathbb{N}$ there exists $C_{k}>0$, such that for all
$X,T_{1},...,T_{k}\in \mathbb{R}^{2n}$ we have 
\begin{equation}\label{inwhk}
    |\sigma^{(k)}(X;T_{1}\otimes\cdots\otimes T_{k})|\leq C_{k}M(X)\prod_{i=1}^{k} g_{X}^{1/2}(T_{i})\,.
\end{equation}
The notation $\sigma^{(k)}$ stands for the $k^{th}$ derivative  of $\sigma$ and  $\sigma^{(k)}(X;T_{1},...,T_{k})$ denotes the $k^{th}$  derivative tensor of $a$ at $X$ in the directions $T_{1},...,T_{k}$. 
For $\sigma\in S(M,g)$ we denote by $\parallel \sigma\parallel_{k,S(M,g)}$ the minimum $C_{k}$ satisfying the above inequality. The class $S(M,g)$ becomes a Fr\'echet space endowed with the family of seminorms $\parallel \cdot\parallel_{k,S(M,g)}$.
\end{definition}


The action of the Weyl quantization on the Sobolev spaces is determined by the following theorem (cf. \cite{BC94}).
\begin{theorem}
\label{thm.cont}
Let $g$ be a H\"{o}rmander metric, and let $M_1,M_2$ be two $g$-weights. Then for any $a \in S(M,g)$ we have that
\[
a^{\tau}(x,D): H(M_1,g) \longrightarrow H(M_1/M,g)\,,
\] admits a bounded extension for all $\tau$.  Since $
H(1,g)=L^2
,$ in the  particular case $M_1=M,$ $a^{\tau} : H(M,g) \longrightarrow L^2,$ admits a bounded extension. 
\end{theorem}

For the general aspects, about the $L^p$ and H\"older boundedness of pseudo-differential operators we refer the reader to Beals \cite{Be,Be2} and to more recent works \cite{Delgado2006,Profe2} and \cite{CardonaHolder}.

\section{The H\"ormander metric adapted to the degenerate Hamiltonian}\label{Metric}
Motivated by the study of the family of operators given by \eqref{daho}, in this section, we will study the general metric \eqref{gen.metric} associated with the bigger class of pseudo-differential operators with symbol as in \eqref{hami1}. In particular, in Lemma \ref{class} we show that $a \in S(m,g)$ for $m,g$ as in the next theorem that applies to the general situation on $\mathbb{R}^n$.



\begin{theorem}\label{thmg1} Let $a_2\in S^2(\Ra^n\times\Ra^n)$ be non-negative $(\geq 0)$. We define the following Riemannian metric on $\Ra^n\times\Ra^n$, 
\begin{equation}\label{gen.metric}
    g_{X}(dx,d\xi):= m^{-1}(x,\xi)(\langle \xi \rangle^{2}+|x|^2)dx^{2}+m^{-1}(x,\xi)d\xi^{2}\,,
\end{equation}
where
\begin{equation}\label{gen.weight}m(x,\xi)=a(x,\xi)+\jpxe,\end{equation}
and the symbol $a$ is given by 
\begin{equation}\label{gen.symbol}
    a(x,\xi)=a_2(x,\xi)+|x|^2\,.
\end{equation}
Then  $g$ is a H\"ormander metric.
\end{theorem}
Before giving the proof, the following inequality will be useful to prove the continuity for the metric $g$ and  it follows by a simple argument using  Taylor's formula of order $2$.%

\begin{proposition}\label{postfi2} Let $f$ be a non-negative function in $C^2(\Ra)$. Then
\[(f'(t))^2\leq \norm{ f''}_{L^{\infty}}f(t), \mbox{ for all }t\in\Ra .\]
\end{proposition}
Also the following lemma will be used to simplify the proof  of continuity and temperateness of the metric.
\begin{lemma}\label{shub1} Let $g_s$ be the Riemannian metric on $\Ra^n\times\Ra^n$ defined by
\[(g_s)_X:= dx^{2}+\frac{d\xi^{2}}{\jp^2+|x|^2}.\]

Then $g_s$ is a H\"ormander metric.
\end{lemma}%
\begin{proof} Since $1\leq \langle \xi\rangle^2 +|x|^2,$ and $g_s^{\sigma}(dx,d\xi)=(\jp^2+|x|^2)dx^{2}+d\xi^2$, the uncertainty principle follows.\\%

 We note that %
  \[(G_s)_X(dx,d\xi):=\frac{dx^{2}}{1+|\xi|^2+|x|^2}+\frac{d\xi^{2}}{1+|\xi|^2+|x|^2}\leq C(g_s)_X.\]%
 Since the weight $1+|\xi|^2+|x|^2$ is continuous with respect to the Shubin metric $G_s$,  it follows that it is also continuous with respect to $g_s$. Hence $g_s$ is continuous.\\
 
To prove the temperateness of $g_s$,  we will see that  
\[\frac{\jp^2+|x|^2}{\langle \eta\rangle^2 +|y|^2} \leq C((g_s)_Y^{\sigma}(X-Y)+1)^N,\] 
 for suitable $C, N$.\\

We observe that \[\frac{\jp^2+|x|^2}{\langle \eta\rangle^2 +|y|^2}=\frac{\jp^2+|x|^2-\langle \eta\rangle^2-|y|^2}{\langle \eta\rangle^2 +|y|^2}+1=\frac{\jp^2-\langle \eta\rangle^2 +|x|^2-|y|^2}{\langle \eta\rangle^2 +|y|^2}+1.\]
 Now
 \begin{equation}
\frac{\jp^2-\langle \eta\rangle^2 }{\langle \eta\rangle^2 +|y|^2}\leq C\frac{(|\xi-\eta|^2+|\eta|^2)}{\langle \eta\rangle^2+|y|^2}\leq C(|\xi-\eta|^2+1)\leq C((g_s)_Y^{\sigma}(X-Y)+1).\label{swcont5}\end{equation}
The other term $\frac{|x|^2-|y|^2}{\jpe^2+|y|^2}$ can be estimated similarly. Thus $g_s$ is temperate and this concludes the proof of the lemma.
\end{proof}%

\begin{proof} (Proof of Theoren \ref{thmg1}.)  The uncertainty principle follows from   
 \[\jp^2+|x|^2\leq (a_2(x,\xi)+|x|^2+\jpxe)^2,\]%
 which holds immediately.\\%
 
 To prove the continuity, we note that %
 \begin{equation}\label{g.product}g_X=\frac{\jp^2+|x|^2}{a_2+|x|^2+\jpxe}\left(dx^2+\frac{d\xi^2}{\jp^2+|x|^2}\right)\geq C'(g_s)_X\geq C(G_s)_X,\end{equation}
 since $\frac{\jp^2+|x|^2}{m(X)}\geq C'$, and where $G_s$ is the Shubin metric as in the proof of Lemma \ref{shub1}. But $\jp^2+|x|^2$ is a continuous weight with respect to $G_s$, so the same holds true with respect to $g$ as well. \\
 
 The continuity of $g$ is now reduced to the study of the ratio $\frac{m(X)}{m(Y)}$, which can be reduced to the analysis of $\,\frac{m(X)-m(Y)}{m(Y)}$.\\
 
 So we first note that %
 \[m(X)-m(Y)=(a_2(X)-a_2(Y))+|x|^2-|y|^2+\jpxe-\jpye.\]
 For the term $a_2(X)-a_2(Y)$, we use the Taylor's inequality of order $2$:
 \[|a_2(X)-a_2(Y)-da_2(Y)\cdot (X-Y)|\leq C\sup_{t\in [0,1]}|d^2a_2(Y+t(X-Y))\cdot (X-Y)^2|\,.\]
 Since  $a_2\in S^2(\Ran\times\Ran)$, we have
 \[\sup_{t\in [0,1]}|d^2a_2(Y+t(X-Y))\cdot (X-Y)^2|\leq  C(\jpe^2|x-y|^2+|\xi-\eta|^2).\]
 On the other hand, by  Proposition \ref{postfi2} applied to the  partial derivatives of $a_2$, we get
 \begin{align*}
 \sup_{t\in [0,1]}|da_2(Y+t(X-Y))\cdot (X-Y)|\leq &  C(|\partial_xa_2(Y)\cdot (x-y)|+|\partial_{\xi}a_2(Y)\cdot (\xi-\eta)|.\\
     \leq & Ca_2^{\half}(Y)\jpe|x-y|+Ca_2^{\half}(Y)|\xi-\eta|.
 \end{align*}
 Therefore  
 \[  |a_2(X)-a_2(Y)|\leq C(a_2^{\half}(Y)\jpe|x-y|+a_2^{\half}(Y)|\xi-\eta|+\jpe^2|x-y|^2+|\xi-\eta|^2)\,. \]
 Now %
 \[\frac{|a_2(X)-a_2(Y)|^2}{m^2(Y)}\leq\]
 \begin{align*}
  \leq &   C\left(\frac{a_2(Y)\jpe^2|x-y|^2}{(a+\jpye)^2}+\frac{a_2(Y)|\xi-\eta|^2}{(a+\jpye)^2}+\frac{\jpe^4|x-y|^4}{(a+\jpye)^2}+\frac{|\xi-\eta|^4}{(a+\jpye)^2}\right)\\
   \leq &   C\left(\frac{(\jpe^2+|y|^2)|x-y|^2}{a+\jpye}+\frac{|\xi-\eta|^2}{a+\jpye}\right)+C\left( \frac{(\jpe^2+|y|^2)^2|x-y|^4}{(a+\jpye)^2}+\frac{|\xi-\eta|^4}{(a+\jpye)^2}\right)\\
 \leq & C(g_Y(X-Y)+(g_Y(X-Y))^2).%
 \end{align*}
Hence
\begin{equation}\label{a1i21}
     \frac{|a_2(X)-a_2(Y)|}{m(Y)}\leq C \left(g_Y(X-Y))^{\half}+g_Y(X-Y)\right).
\end{equation}
 
 For the term $\frac{\jpxe-\jpye}{m(Y)}$, we note that $|\jpxe-\jpye|\leq C(||\xi|-|\eta||+||x|-|y||)$. We have
 \[\frac{(|\xi|-|\eta|)^2}{m^2(Y)}\leq  \frac{|\xi-\eta|^2}{m^2(Y)}\leq \frac{|\xi-\eta|^2}{a+\jpye}\leq Cg_Y(X-Y). \]
 We also have
 \[\frac{(|x|-|y|)^2}{m^2(Y)}\leq  \frac{|x-y|^2}{m^2(Y)}\leq \frac{(\jpe^2+|y|^2)|x-y|^2}{a+\jpye}\leq Cg_Y(X-Y). \]
 Thus 
\begin{equation}\label{aux23}
 \frac{|\jp-\jpe|}{m(Y)}\leq C(g_Y(X-Y))^{\half}.
 \end{equation}
The term $\frac{||x|^2-|y|^2|}{m(Y)}$ can be analysed in  a similar way, obtaining 
\begin{equation}\label{hy78}
  \frac{||x|^2-|y|^2|}{m(Y)}\leq C \frac{|x-y|^2+|y|^2}{m(Y)}\leq C\left(\frac{(\jpe^2+|y|^2)|x-y|^2}{m(Y)} +1\right) =   C(g_Y(X-Y)+1).
\end{equation}
 By  \eqref{a1i21}, \eqref{aux23} and \eqref{hy78}, we get
\begin{equation}\label{mctw1}
\frac{m(X)}{m(Y)}\leq C(g_Y(X-Y)+g_Y^{\half}(X-Y))+1)\leq C(g_Y(X-Y)+1).
\end{equation}

 Therefore, the metric $g$ is continuous.\\%

In order to prove the temperateness of $g$, we first observe that from the proof of Lemma \ref{shub1}, by using the first inequality in \eqref{swcont5}, we obtain  
\begin{eqnarray*}
\frac{|\jp^2-\langle \eta\rangle^2| }{\langle \eta\rangle^2 +|y|^2}\leq & C\left(\frac{|\xi-\eta|^2}{\langle \eta\rangle^2 +|y|^2}+1\right)\\
\leq & C  \left(\frac{a(Y)+\jpye}{\langle \eta\rangle^2 +|y|^2}|\xi-\eta|^2+1\right)\\
\leq & C(g_Y^{\sigma}(X-Y)+1).
\end{eqnarray*}
In a similar way, for the term $\frac{|x|^2-|y|^2}{\jpe^2+|y|^2}$ we have
\begin{eqnarray*}
\frac{||x|^2-|y|^2|}{\jpe^2+|y|^2}\leq & C\left(\frac{|x-y|^2}{\langle \eta\rangle^2 +|y|^2}+1\right)\\
\leq & C\left((a(Y)+\jpye)|x-y|^2+1\right)\\
\leq &C(g_Y^{\sigma}(X-Y)+1).
\end{eqnarray*}%

On the other hand, for the ratio $\frac{m(X)}{m(Y)}=\frac{a(X)+\jpX}{a(Y)+\jpye}$, we note that by \eqref{mctw1} and since $g\leq g^{\sigma}$ , we get
\begin{equation}\label{mh54s}
\frac{m(X)}{m(Y)}\leq C(g_Y(X-Y)+1)\leq C(g_Y^{\sigma}(X-Y)+1).\end{equation}
Therefore $g$ is temperate and we conclude the proof.
\end{proof}

As a consequence of \eqref{mh54s} we have that  $m$ is indeed a $g$-weight.
\begin{corollary}
The function $m$ given in \eqref{gen.weight} is a $g$-weight with respect to the metric $g$ defined by \eqref{gen.metric}.
\end{corollary}
\begin{lemma}\label{class:smg} Let us consider the H\"ormander metric $g$ defined by \eqref{gen.metric} and the $g$-weight in \eqref{gen.weight}. Then, a smooth symbol $\sigma:=\sigma(x,\xi)\in S(m,g)$ if and only if
\begin{equation}
    |\partial_x^{\beta}\partial_{\xi}^{\alpha}\sigma(x,\xi)|\leq C_{\alpha\beta}m^{1-\frac{|\alpha|}{2}-\frac{|\beta|}{2}}(X)(\jp^2+|x|^2)^{\frac{|\beta|}{2}}.
\end{equation}
\end{lemma}
\begin{proof} The H\"ormander metric is given by
\begin{equation}
    g_{X}(dx,d\xi):= m^{-1}(x,\xi)(\langle \xi \rangle^{2}+|x|^2)dx^{2}+m^{-1}(x,\xi)d\xi^{2}.
\end{equation}The class  $S(m,g)$ consists of all smooth functions $\sigma\in C^\infty(\mathbb{R}^{2n})$ such that for any $k \in \mathbb{N}$ there exists $C_{k}>0$, such that for all
$X,T_{1},...,T_{k}\in \mathbb{R}^{2n}$ we have that
\begin{equation}\label{Proof:lemma:metric}
    |\sigma^{(k)}(X;T_{1}\otimes\cdots\otimes T_{k})|\leq C_{k}m(X)\prod_{i=1}^{k} g_{X}^{1/2}(T_{i}).
\end{equation} In consequence, $\sigma\in S(m,g),$ if and only if, for every $\alpha\in \mathbb{N}_0^n$ and for every $\beta\in \mathbb{N}_0^n,$ we have that
\begin{align*}
    |\partial_x^\beta\partial_\xi^\alpha \sigma(x,\xi)|\leq C_{\alpha,\beta}m(x,\xi)\left(m^{-1}(x,\xi)(\langle \xi \rangle^{2}+|x|^2)\right)^{\frac{|\beta|}{2}}(m^{-1}(x,\xi))^{\frac{|\alpha|}{2}},
\end{align*}since any partial derivative with respect to $x,$ in any canonical direction $x_i,$ contributes on the right hand side of \eqref{Proof:lemma:metric} with the factor
$$ \left(m^{-1}(x,\xi)(\langle \xi \rangle^{2}+|x|^2)\right)^{\frac{1}{2}}. $$ Also, any partial derivative with respect to $\xi,$ in any canonical direction $\xi_j,$ contributes on the right hand side of \eqref{Proof:lemma:metric} with the factor
$$ \left(m^{-1}(x,\xi))\right)^{\frac{1}{2}}. $$ Observing that
\begin{align*}
   m(x,\xi)&\left(m^{-1}(x,\xi)(\langle \xi \rangle^{2}+|x|^2)\right)^{\frac{|\beta|}{2}}m^{-1}(x,\xi)^{-\frac{|\alpha|}{2}}\\
   &=m(x,\xi)^{1-\frac{|\beta|}{2}-\frac{|\alpha|}{2}}(\langle \xi \rangle^{2}+|x|^2)^{\frac{|\beta|}{2}},
\end{align*}we conclude the proof of Lemma \ref{Proof:lemma:metric}.   
\end{proof}

We will now restrict our attention to non-negative symbols $a_2$ of differential operators of order $2$. 
\begin{definition} Let $C_b^{\infty}(\Ran)$ be the class of $C^{\infty}$ functions over $\Ran$ with bounded derivatives of any order. We denote by $\mbox{Diff }_+^2(\Ran)$ the class of differential operators of order $2$ on $\Ran$ with $C_b^{\infty}(\Ran)$ coefficients and non-negative symbol. 
\end{definition}
\begin{lemma}\label{class} Let  $a_2(x,D)\in {\mbox{Diff  }}_+^2(\Ran)$. We consider $a(x,\xi)=a_2(x,\xi)+|x|^2$ and the corresponding H\"ormander metric $g$ and the weight $m$ as in \eqref{gen.metric} and \eqref{gen.weight}, respectively. Then  $a,\, m \in S(m,g)$.
\end{lemma}%
\begin{proof} We first consider the membership of $a$ in $S(m,g)$.  We  want to prove that for all $k\in\mathbb{N} $ there exists $C_k>0$ such that for all $X, T_1,\cdots, T_k \in \Ran\times\Ran$:
\[|(a^{(k)}(X); T_1\otimes T_2\otimes\cdots\otimes T_k)|\leq C_{k}m(X)\prod\limits_{j=1}^kg_X^{\half}(T_j).\]
As it is customary these estimates are reduced to the ones for canonical directions which are equivalent to the following symbol inequalities (see Lemma \ref{class:smg}):
 for all $\alpha,\beta$, there exists $C_{\alpha\beta}>0$ such that
\[|\partial_x^{\beta}\partial_{\xi}^{\alpha}a(x,\xi)|\leq C_{\alpha\beta}m^{1-\frac{|\alpha|}{2}-\frac{|\beta|}{2}}(X)(\jp^2+|x|^2)^{\frac{|\beta|}{2}}.\]
We are also going to prove that $m\in S(m,g).$ 
We write $$T=(t,\tau)=(t_1,\dots, t_n,\tau_1,\dots, \tau_n)\in\Ran\times\Ran.$$ We observe that from the definition of the metric $g$ we have 
\[\frac{(\jp^2+|x|^2)}{m(X)}|t|^2\leq g_X(T);\hspace{0.4cm} \frac{|\tau|^2}{m(X)}\leq g_X(T).\]
Then%
\[\frac{|t_i|}{g_X^{\half}(T)}\leq \frac{m^{\half}(X)}{(\jp^2+|x|^2)^{\half}};\hspace{0.4cm} \frac{|\tau_j|}{g_X^{\half}(T)}\leq m^{\half}(X).\]%
For $k=1$ we will apply Proposition \ref{postfi2} to the partial derivatives of the symbol $a_2\in S^2(\Ran\times\Ran)$. Since  $\partial_{x_i}a(X)=\partial_{x_i}a_2(X)+2x_i$, we have
\begin{eqnarray*}
    |\partial_{x_i}a(X)|\frac{1}{m(X)}\frac{m^{\half}(X)}{(\jp^2+|x|^2)^{\half}} & \leq & C\frac{a_2^{\half}(X)\jp+|x_i|}{m^{\half}(X)(\jp^2+|x|^2)^{\half}}\\
    & \leq & C\frac{a_2^{\half}(X)\jp}{m^{\half}(X)\jp}+C\frac{|x_i|}{(1+|x|^2)^{\half}} \leq C.
\end{eqnarray*}

On the other hand, since $\partial_{\xi_j}a(X)=\partial_{\xi_j}a_2(X)$ we have%
\[|\partial_{\xi_j}a(X)|\frac{1}{m(X)}m^{\half}(X)\leq C\frac{a_2^{\half}(X)}{m^{\half}(X)}\leq C.\]%
If $k=|\alpha|+|\beta|\geq 2$, we will use the membership of $a_2$ to the class $S^2(\Ran\times\Ran)$. \\

We first observe that if $|\alpha|=0$,  then $|\beta|\geq 2$ and 
\[|\partial_x^{\beta}\partial_{\xi}^{\alpha}a(x,\xi)|=|\partial_x^{\beta}(a_2(x,\xi)+|x|^2)|\leq C_1\jp^2+C_2\leq C\jp^2.\]%

If $0<|\alpha|\leq 2$, we have 
\[|\partial_x^{\beta}\partial_{\xi}^{\alpha}a(x,\xi)|=|\partial_x^{\beta}\partial_{\xi}^{\alpha}a_2(x,\xi)|\leq C\jp^{2-|\alpha|}.\]%
Thus, for $0\leq|\alpha|\leq 2$, we obtain%
\begin{eqnarray*}
     |\partial_x^{\beta}\partial_{\xi}^{\alpha}a(x,\xi)| & \leq & C_{\alpha\beta}\jp^{2-|\alpha|}\\
    & \leq & C_{\alpha\beta} (\jp^2+|x|^2)^{1-\frac{|\alpha|}{2}}\\
   & =  & C_{\alpha\beta }\,\, m(X)(\jp^2+|x|^2)^{\frac{|\beta|}{2}}m^{-\frac{|\beta|}{2}-\frac{|\alpha|}{2}}(\jp^2+|x|^2)^{1-\frac{|\beta|}{2}-\frac{|\alpha|}{2}}m^{-1+\frac{|\beta|}{2}+\frac{|\alpha|}{2}}\\
  & \leq & C_{\alpha\beta }\,\, m(X)(\jp^2+|x|^2)^{\frac{|\beta|}{2}}m^{-\frac{|\beta|}{2}-\frac{|\alpha|}{2}},
\end{eqnarray*}

where in the last inequality we have used the fact that 
$\left(\frac{m(X)}{\jp^2+|x|^2}\right)^{\frac{|\alpha|+|\beta|-2}{2}}\leq C$, since $m(X)\leq C(\jp^2+|x|^2).$ This gives the right estimate for $0\leq|\alpha|\leq 2$.\\%

For the case $|\alpha|\geq 3$, we note that since $a_2(x,D)$ is a differential operator of order $2$ we have \[\partial_{\xi}^{\alpha}a(x,\xi)=\partial_{\xi}^{\alpha}a_2(x,\xi)=0.\]
Hence $\partial_x^{\beta}\partial_{\xi}^{\alpha}a(x,\xi)=0$, 
and  the proof is concluded.\\%

 In order to see that  $\jpX\in S(m,g)$, we observe that  $\jpX$ is a regular weight with respect to the metric $G_s$. Then
 \begin{equation}\label{shX12}\jpX\in S(\jpX, G_s)\subset S(\jpX, g)\subset S(m, g).\end{equation}%
 We have also used the fact that $\jpX$ is a $g-$weight.\\%
 
 Therefore $m=a+\jpX\in S(m,g)$.
 \end{proof}%

In the special  case  of $a_2(x,\xi)=\xi_{1}^{2}+\tilde{x_1}^{2}\xi_{2}^{2}$, consequently we have the following property.
 \begin{corollary}\label{thm.Hor.m}
The metric $g$ given in \eqref{gen.metric}  with  $a_2(x,\xi)=\xi_{1}^{2}+\tilde{x_1}^{2}\xi_{2}^{2}$,  is a H\"{o}rmander metric on $\Ra^2\times\Ra^2$,  $m(X)=\xi_{1}^{2}+\tilde{x_1}^{2}\xi_{2}^{2}+|x|^2+\jpxe$ is a $g$-weight and $\xi_{1}^{2}+\tilde{x_1}^{2}\xi_{2}^{2}+|x|^2,\,\,  m=\xi_{1}^{2}+\tilde{x_1}^{2}\xi_{2}^{2}+|x|^2+\jpxe\in S(m,g).$
\end{corollary}%

\section{Some applications to well-posedness for degenerate Schr\"odinger and degenerate parabolic equations}\label{applicatione}
We are now going to establish  some implications in the  $L^2$ and Sobolev spaces $H(M,g)$ context. Indeed, the construction of the metric $g$ and  the class $S(m,g)$ adapted to our degenerate harmonic oscillators  has other  consequences, regarding the well-posedness for degenerate Schr\"odinger equations and  degenerate parabolic equations in the  $L^2$-setting.  We introduce the following appropriate class of potentials on $\Ran$.
\begin{definition} We will denote by  $\mathcal{P}_2(\Ran)$, the class of Borel functions $V:\Ran\rightarrow \Ra$ satisying the following conditions: 
\begin{eqnarray*}  
\mbox{(V1)}    & \mbox{There exist }C, C_1>0 \mbox{ such that }  |V(x)|\leq C|x|^2, \mbox{ for a.e. } |x|\geq C_1. \\
 \mbox{(V2)} & \mbox{There exists } C_2>0 \mbox{ such that } V(x)\geq -C_2, \mbox{ for a.e. } x\in\Ran .
\end{eqnarray*}
\end{definition}

It is worth to point out that no regularity assumption has  been imposed on $V$ as is customary in most of the  main theorems on the classical spectral theory of Schr\"odinger operators where the conditions are more likely to be on the integrability (\cite{BeShub}, \cite{HisSigal}).  We will consider Hamiltonians of the form $\mathcal{H}_V=a_2(x,D)+V$. 
\begin{theorem}\label{CatSchrodinger} Let $a_2(x,D)\in  \mbox{Diff }_+^2(\Ran)$ be  formally self-adjoint, $V\in \mathcal{P}_2(\Ran)$,  $g$  the H\"ormander metric and $m$ the $g$-weight associated to $a_2(x,D)$ as in \eqref{gen.metric}. Then, $-i\mathcal{H}_V$ is the infinitesimal generator of a $C_0-$group of unitary operators. Consequently, the following Cauchy problem for the corresponding degenerate Schr\"odinger equation is well-posed on $L^2$:
\begin{equation}\label{EQ:cpxaax11}
\left\{\begin{array}{rl}
i\partial_t u &=  \mathcal{H}_Vu, \\
u(0)&= f.
\end{array}
\right.
\end{equation}%
Furthermore, if $a_2(x,D)=-\sum\limits_{j=1}^k X_j^2$ is a sum of squares of real vector fields, then  $\mathcal{H}=-(a_2(x,D)+|x|^2)$ is the infinitesimal generator of a $C_0-$semigroup (degenerate harmonic semigroup) of contractions and the corresponding degenerate parabolic equation is well-posed on $L^2$: 
\begin{equation}\label{EQ:cpxaax11he}
\left\{\begin{array}{rl}
\partial_t u &=  \mathcal{H}u, \\
u(0)&= f.
\end{array}
\right.
\end{equation}%
\end{theorem}

\begin{proof} We first state the boundedness of the potential operator. Indeed, we first note that since $|x|^2\in S(m,g)$, we have that $|x|^2:H(m,g)\rightarrow L^{2}(\Ran)$ is bounded. From (V1) and (V2), we have $|V(x)|\leq C(1+|x|^2)$, for a.e. $x\in\Ran$. Since $m\geq 1$, it is now clear that  \begin{equation}\label{Vbded1}\lvert|V(x)u\rvert|_{L^2}\leq C\lvert|u\rvert|_{H(m,g)},\end{equation}
for all $u\in H(m,g)$.\\%

On the other hand, the operator $a_2(x,D)+V(x)$ is symmetric, Lemma \ref{class}, Theorem \ref{thm.cont} and the inequality \eqref{Vbded1} imply that the operator $a_2(x,D)+V(x):H(m,g)\rightarrow L^{2}(\Ran)$ is bounded. Moreover, since $0\leq a_2(x,\xi)\in S^2(\Ran\times\Ran)$, applying Fefferman-Phong theorem
 there exists $C>0$ such that 
\[\mbox{Re}\langle a_2(x,D)v,v \rangle\geq -\frac{C}{2}\lvert | v\rvert |_{L^2(\Ran)}^2. \]%
Hence%
\[\mbox{Re}\langle (a_2(x,D)+C)v,v \rangle\geq \frac{C}{2}\lvert | v\rvert |_{L^2(\Ran)}^2.\]%
Thus, by using (V2) there exists $\widetilde{C}>0$ such that  %
\[\mbox{Re}\langle (a_2(x,D)+V(x)+\widetilde{C})v,v \rangle\geq \frac{C}{2}\lvert | v\rvert |_{L^2(\Ran)}^2,\]%
and by the Cauchy-Schwarz inequality  we have
\begin{equation}\label{aprioriw34}
    \lvert |(a_2(x,D)+V(x)+\widetilde{C})v\rvert |_{L^2(\Ran)}\geq C \lvert |v\rvert |_{L^2(\Ran)}.
\end{equation}%
Therefore, Friedrichs extension theorem ensures that $\mathcal{H}_{V}+\widetilde{C}:H(m,g)\rightarrow L^{2}(\Ran)$ is self-adjoint. Thus, $\mathcal{H}_{V} $ is also self-adjoint. Now, an application of Stone's theorem guarantees that  $-i\mathcal{H}_V$ is the infinitesimal generator of a $C_0-$group of unitary operators. Therefore, the  Cauchy problem for the corresponding degenerate Schr\"odinger equation \eqref{EQ:cpxaax11} is well-posed in $L^2$.\\

On the other hand since $\mathcal{H}=-(a_2(x,D)+|x|^2)$ is self-adjoint, $-\mathcal{H}+C$ is semibounded for some $C>0$ as above with $V=|x|^2$, and $-a_2(x,D)$ being a sum of squares of real vector fields,  $\mathcal{H}$ is dissipative. Then, the Lummer-Phillips theorem ensures that $\mathcal{H}$ is the infinitesimal generator of a $C_0-$semigroup of contractions and hence the corresponding degenerate parabolic equation \eqref{EQ:cpxaax11he} is well-posed. 
 \end{proof}
\begin{remark} We recall that the simultaneous study of the Hamiltonian as an infinitesimal generator for a Schr\"odinger semigroup as  
for  Heat semigroup has been an active field of research in the last decades as it  was well pointed out by B. Simon in \cite{BSimon1}.  For instance, when looking at the $L^p$ properties of the eigenfunctions of $\mathcal{H}=-\Delta+ V(x)$, if $\mathcal{H}\psi=E\psi$ and $\psi\in L^2$, one observes that since $e^{-t\mathcal{H}}\psi=e^{-tE}\psi$, then $\psi\in e^{-t\mathcal{H}}(L^2)$. Thus, a mapping property such as $e^{-t\mathcal{H}}:L^2\rightarrow L^{\infty}$ can be used to deduce  that $\psi$ belongs also to $ L^{\infty}$. The works of B. Simon and R. Carmona  on Schr\"odinger semigroups have deeply influenced the research on these semigroups,  their  applications are diverse and the literature is currently huge.
See for instance \cite{albev:hs}, \cite{Carmona1}, \cite{BSimon2},  \cite{BSimon3}, \cite{iou:a1s}, \cite{ishik:a1},  for more accounts on this matter. Thus, the Schr\"odinger semigroups are relevant for the study of $L^p$ properties of eigenfunctions of Schr\"odinger operators and especially for the so-called {\it ground states}.
\end{remark}

As a corollary of our previous $L^2$ estimates, we can extend  the well-posedness  for the positive powers of $\mathcal{H}$, in the case  where $$a_2(x,D)=-\sum\limits_{j=1}^k X_j^2$$ is a sum of squares of real vector fields, where the family of vector fields satisfies the H\"ormander condition of order $2$. A reason for considering a sub-Laplacian associated  to a system of vector fields satisfying  the H\"ormander condition of order 2 is that in this case the geometry of the balls associated to the Carn\'ot-Carath\'eodory distance of the H\"ormander system is very well understood from the fundamental work of Nagel, Stein and Wainger \cite{NagelSteinWainger}, see Remark \ref{Subelliptic:remark}. 

On the other hand, thanks to the Fefferman-Phong inequality, we can define the positive powers of $\mathcal{H}^{\beta}$, and we have the membership $\mathcal{H}^{\beta}\in OpS(m^{\beta},g)$. Moreover, the maximal hypoellipticity of $a_2(x,D)$ guarantees that the Sobolev space $ H(m^{\beta},g)$ is characterised by the condition $(a_2(x,D)+|x|^2+C)^{\beta}u\in L^2$. Hence,  $(a_2(x,D)+|x|^2+C)^{\beta}$ is semi-bounded from below 
 and consequently we obtain the following corollary.
\begin{corollary}\label{CatSchrodinger2} Let $\beta>0.$ Let $$a_2(x,D)=-\sum\limits_{j=1}^k X_j^2$$ be  a sum of squares of real vector fields, where $X_1,\dots, X_k$ satisfy the H\"ormander condition of order $2$, $g$ the H\"ormander metric and $m$ the $g$-weight associated to $a_2(x,D)$ as in \eqref{gen.metric}. Then, there exists $C>0$ such that  $\mathcal{H}_{C}=a_2(x,D)+|x|^2+C$ is semibounded in $L^2(\Ran)$, and $-i\mathcal{H}_C$ is the infinitesimal generator of a $C_0-$semigroup. Consequently, the following Cauchy problem for the corresponding degenerate Schr\"odinger equation is well-posed on $L^2$:
\begin{equation}\label{EQ:cpxddaax11}
\left\{\begin{array}{rl}
i\partial_t u +  \mathcal{H}_0^{\beta}u & =0, \\
u(0)&= f.
\end{array}
\right.
\end{equation}%
Furthermore, the operator $$\mathcal{H}=-(a_2(x,D)+|x|^2),$$ is the infinitesimal generator of a $C_0-$semigroup (degenerate harmonic semigroup) of contractions and the corresponding degenerate parabolic equation is well-posed on $L^2$: 
\begin{equation}\label{EQ:cpxaaxdd11he}
\left\{\begin{array}{rl}
\partial_t u -  \mathcal{H}^{\beta}u & =0, \\
u(0)&= f.
\end{array}
\right.
\end{equation}%
\end{corollary}
 

 In order to go further with our analysis we need the property of  {\it geodesic temperateness} for our metric $g$. This property allows us to obtain the invertibilitiy of the Hamiltonian as a pseudodifferential operator. This important notion was introduced by Bony and Chemin in \cite{BC94}, originally calling it {\it strong temperateness}. We start by recalling its definition, see \cite{le:book}, \cite{Bonygtemp}.\\
 
  Let $d^{\sigma}(X,Y)$ be the geodesic distance for the Riemannian  metric $g^{\sigma}$ between $X$ and $Y$, where $g^{\sigma}$ is given by \eqref{gsigma}.
  \begin{definition} A H\"ormander metric $g$ is called {\it geodesically temperate} if it is temperate and if there exist $C, N>0$ such that
  \begin{equation}\label{geode} \frac{g_Y}{g_X}\leq C(1+d^{\sigma}(X,Y))^N.
        \end{equation}
  It is known (c.f. \cite{Bonygtemp}) that the condition \eqref{geode}  is equivalent to : there exist $C, N>0$ such that
\begin{equation}\label{geode2} C^{-1}(1+d^{\sigma}(X,Y))^{\frac{1}{N}}\leq (1+g^{\sigma}(X-Y))\leq C(1+d^{\sigma}(X,Y))^{N}.
        \end{equation}
 \end{definition}
  Since our H\"ormander metric is symmetrically split, it follows that it is geodesically temperate by a result of Bony (Cf. Theorem 5 (i), \cite{Bonygtemp}). 
 In what follows we require the following definition.  We record that a second order differential operator $P(x,D)$ is subelliptic of order $0<\tau\leq 2,$ if it satisfies the {\it a-priori estimate} \begin{equation}\label{sub}
    \lVert v\rVert_{H^{\tau}}\leq C_1(\lVert P(x,D)v\rVert_{L^{2}}+\lVert v\rVert_{L^{2}}),\,\,\,\forall v \in C_{0}^{\infty}(\Ran),
\end{equation}where $H^{\tau}$ denotes the standard Sobolev space of order $\tau$ on $\mathbb{R}^n.$
 \begin{corollary}  Let $a_2(x,D)\in  \mbox{Diff }_+^2(\Ran)$ be  formally self-adjoint and subelliptic with order of subellipticity $0<\tau \leq 2$, $g$  the H\"ormander metric and $m$ the $g$-weight associated to $a_2(x,D)$ as in \eqref{gen.metric}. Then, there exist $C>0$ such that the  operator  
\[\mathcal{H}_{C}=a_2(x,D)+|x|^2+C:H(m,g)\rightarrow L^2(\Ran),\]
    is an isomorphism with inverse a pseudodifferential operator
 \[\mathcal{H}_{C}^{-1}:L^2(\Ran)\rightarrow H(m,g),\]
 with symbol in $S(m^{-1},g)$.
\end{corollary} 
\begin{proof} The {\it a priori} estimate \eqref{aprioriw34} and the self-adjointness of $\mathcal{H}_{C}$ gives the surjectivity. On the other hand, the subellipticity of $a_2(x,D)$ gives
\begin{equation}\label{sub}
    \lVert v\rVert_{H^{\tau}}\leq C_1(\lVert(a_2(x,D)+|x|^2)v\rVert_{L^{2}}+\lVert v\rVert_{L^{2}}),\,\,\,\forall v \in C_{0}^{\infty}(\Ran).
\end{equation}

By the  {\it a priori} estimate \eqref{aprioriw34} again and \eqref{sub}, there exists a constant $C_2>0$ such that
\begin{equation} \label{sub2}\lVert u\rVert_{H^{\tau}}\leq C_2\lVert \mathcal{H}_{C}u\rVert_{L^{2}},\,\,\,\forall u\in C_{0}^{\infty}(\Ran).\end{equation}
Now, an analogous argument to the one for the proof of Lemma 3.3 in \cite{subedel1}, splitting the phase-space within an elliptic and a subelliptic zones with respect to $g$ give us 
\[\Vert u\Vert_{H(m,g)}\leq C_2\Vert \mathcal{H}_{C}u\Vert_{L^{2}},\,\,\,\forall u\in C_{0}^{\infty}(\Ran).\]
Hence $\mathcal{H}_{C}:H(m,g)\rightarrow L^2(\Ran)$ is injective. Therefore, an application of Corollary 7.7 in \cite{BC94} concludes the proof.
\end{proof}

\section{$L^p$-bounds, $H^1$-$L^1$ and $L^{\infty}$-$BMO$ estimates for degenerate operators}\label{boundednnesssection}%

In this section we consider our analysis of pseudo-differential operators associated to the degenerate class of operators $S(m,g)$ where $g$ is the metric in \eqref{gen.metric}.
\subsection{$L^\infty$-$BMO$ and $H^1$-$L^1$ boundedness}
The next lemma will be essential for us in order to estimate an optimal fractional power of the weight $m(x,\xi)$, which determines a suitable class for the $L^p$-boundedness. The $L^{\infty}$-$BMO$ boundedness will be reached by analysing a partition of unity of the symbol of interest. The estimates for each piece is obtained from the next lemma which is an extension of a classical Hardy-Littlewood-Sobolev inequality. 
\begin{lemma}\label{lem:l} Let $\varepsilon_0=\frac{3}{4}$ and  $\varepsilon_0\leq\varepsilon<1,$ and let $q\in S(m^{-\frac{n}{2}\varepsilon },g)$ be supported in 
    $R\leq a(x,\xi)+\jpxe\leq 3R $ for $R>1 $. Then, for all $l>n/4 $ we have 
\[\lvert | q(x,D)f \rvert|_{\infty}\leq C \lvert |q\rvert |_{l; S(m^{-\frac{n \varepsilon}{2} },g)}  \lvert |f\rvert |_{L^\infty}, \] 
where the constant $C$ is independent of $q$ and $f$.
\end{lemma}

\begin{proof}  Let $q\in S(m^{-\frac{n}{2}\varepsilon },g)$ be supported in 
 \[\{(x,\xi)\in \Ran\times\Ran \, : R\leq a(x,\xi)+\jpxe\leq 3R \}, \]
 for 
    $R>1 $ fixed. We observe that
\begin{align*}
 q(x,D)f(x)=&(2\pi)^{-n}\int\int
  e^{i(x-y)\cdot\xi}q\left(x,\xi\right)f(y)dyd\xi\\
=&\int\hat{q}_x(y-x)f(y)dy\\
=&\hat{q}_x*f(x),
\end{align*}
where $\hat{q}_x$ is the Fourier transform
$q_x(\xi)=q(x,\xi)$ with respect to the variable $\xi$.
 We obtain
\[|q(x,D)f(x)|\leq
    \vert|\hat{q}_x \rvert|_{L^1}  \vert|f\rvert|_{L^\infty},\,\,x\in\Ran. \]
 It is then enough to prove that for all $x\in\Ran$ and $ l>n/4 $
\[ \vert|\hat{q}_x \rvert|_{L^1}\leq C \lvert|q\rvert|_{l; S(m^{-\frac{n}{2}\varepsilon },g)}.\]

\indent Let us consider  $b>0$ to be fixed below.  By applying the
 Cauchy-Schwarz inequality, we obtain
\begin{eqnarray*}
\int\limits_{|y|<b}|\hat{q}_x(y)|dy & \leq & Cb^{\frn}\left(\,\int\limits_{|y|<b}|\hat{q}_x(y)|^{2}dy\right)^{\half}\\
& \leq & Cb^{\frn}\left(\,\int\limits_{\Ran}|q(x,\xi)|^{2}d\xi\right)^{\half}\\
& \leq & C\vert|q\rvert|_{l;S(m^{-\frac{n}{2}\varepsilon },g)}b^{\frn}\left(\,\int\limits_{a(x,\xi)+\langle x, \xi\rangle\leq 3R}m^{-n\varepsilon}(x,\xi)d\xi\right)^{\half}\\
& \leq & C\vert|q\rvert|_{l;S(m^{-\frac{n}{2}\varepsilon },g)}b^{\frn}\left(\,\int\limits_{\jp\leq 3R}\jp^{-n\varepsilon}d\xi\right)^{\half}\,.
\end{eqnarray*}
Now, by choosing $b=R^{\varepsilon -1}$, we obtain
\[\int\limits_{|y|<b}|\hat{q}_x(y)|dy\leq C\vert|q\rvert|_{l;S(m^{-\frac{n}{2}\varepsilon },g)}.\]
On the other hand, for the analysis of the integral $\int\limits_{|y|\geq b}|\hat{q}_x(y)|dy, $ we first note that for an integer $k>\frac{n}{4}$, we have $k>\frac{n}{4}\geq n(1-\varepsilon)$, since $\varepsilon\geq \frac{3}{4}$. Hence
\begin{eqnarray*}
\int\limits_{|y|\geq b}|\hat{q}_x(y)|dy \leq & Cb^{\frn -k }\left(\,\int\limits_{y\geq b}|y|^{2k} |\hat{q}_x(y)|^{2}dy\right)^{\half}\\
\leq & Cb^{\frn -k }\left(\,\int\limits_{\Ran}|\nabla_{\xi}^k q(x,\xi)|^{2}d\xi\right)^{\half},
\end{eqnarray*}and the Plancherel theorem implies the inequality
\begin{equation}\label{order:inequality:lemma}
  \int\limits_{|y|\geq b}|\hat{q}_x(y)|dy   \leq  C\lvert|q\rvert|_{k;S(m^{-\frac{n \varepsilon}{2}},g)}b^{\frn -k }\left(\,\int\limits_{R\leq   a_2(x,\xi) +|x|^2+\jpxe\leq 3R}m^{-n\varepsilon -k}(x,\xi)d\xi\right)^{\half}.
\end{equation}

In order to estimate the last integral, we first note that 
there exists $C>1$ such that for all $(x,\xi)$:
\[a_2(x,\xi) +|x|^2+\jpxe\leq C(\jp^2+|x|^2).\]
We split the argument into two cases
:
$|x|^2\leq \frac{R}{2C}$ and $|x|^2>\frac{R}{2C}$. \\

If $|x|^2\leq \frac{R}{2C}$, we get \[R\leq   a_2(x,\xi) +|x|^2+\jpxe\leq C(\jp^2 + |x|^2)\leq C\jp^2+ \frac{R}{2}. \]
Hence $\frac{R}{2}\leq C\jp^2.$
Thus 
\[b^{\frn -k }\left(\,\int\limits_{R\leq   a_2(x,\xi) +|x|^2+\jpxe\leq 3R}m^{-n\varepsilon -k}(x,\xi)d\xi\right)^{\half}\]

\[\leq  Cb^{\frn -k }\left(\,\int\limits_{R\leq   2C \jp^2}\jp^{-n\varepsilon -k}d\xi\right)^{\half}\leq CR^{(\varepsilon-1)(\frn-k)+\frac{1}{4}(n(1-\varepsilon)-k)}\leq C.\]

To see the last inequality, we note that 
\[(\varepsilon-1)(\frn-k)+\frac{1}{4}(n(1-\varepsilon)-k)=\frac{n}{4}(\varepsilon -1)+k(\frac{3}{4}-\varepsilon).\]
Since $\varepsilon\geq \frac{3}{4}$, we obtain the desired bound.\\

On the other hand, if $|x|^2>\frac{R}{2C}$, we have
$-|x|^2<-\frac{R}{2C}$ and 
\[\jp\leq   a_2(x,\xi) +\jpxe\leq  3R-|x|^2\leq 3R-\frac{R}{2C}=\frac{(6C-1)}{2C}R. \]
Since $6C-1>0$ and by using the fact that $|x|^2>\frac{R}{2C}$, we have $\jpx^{-k}\leq C'R^{-\frac{k}{2}}$ and  we  obtain
\begin{align*}
b^{\frn -k }\left(\,\int\limits_{R\leq   a_2(x,\xi) +|x|^2+\jpxe\leq 3R}m^{-n\varepsilon -k}(x,\xi)d\xi\right)^{\half}=&b^{\frn -k } \left(\,\int\limits_{\jp\leq C'R}m^{-n\varepsilon}(x,\xi)m^{-k}(x,\xi)d\xi\right)^{\half}\\
\lesssim &R^{(\frac{n}{2}-k)(\varepsilon -1)}\left(\,\int\limits_{\jp\leq C'R}\jp^{-n\varepsilon}\jpx^{-2k}d\xi\right)^{\half}\\
\lesssim &R^{(\frac{n}{2}-k)(\varepsilon -1)}\jpx^{-k}\left(\,\int\limits_{\jp\leq C'R}\jp^{-n\varepsilon}d\xi\right)^{\half}\\
\leq & CR^{(\frac{n}{2}-k)(\varepsilon -1)}R^{-\frac{k}{2}}R^{\frac{n}{2}(1-\varepsilon)}.
\end{align*}
By observing that 
\[R^{(\frac{n}{2}-k)(\varepsilon -1)+\frac{n}{2}(1-\varepsilon)-\frac{k}{2}}=R^{k(\frac{1}{2}-\varepsilon)}\leq C,\]
since $\varepsilon\geq\frac{3}{4}\geq \half$, the desired estimate follows.
\end{proof}

 We can now state the $L^{\infty}$-$BMO$ bounds as a consequence of our $L^{\infty}$-$L^{\infty}$ estimates. 
\begin{theorem}\label{parta} Let $\frac{3}{4}\leq \beta<1,$ and let $\sigma(x,D)\in \textnormal{Op}(S(m^{-\frn\beta},g)).$  Then, for  $\ell>\frac{n}{4},$ we have
\begin{equation*}
    \Vert \sigma(x,D)f\Vert_{BMO(\mathbb{R}^n)}\leq C\Vert \sigma(x,\xi)\Vert_{\ell,S(m^{-\frn\beta},g)}\Vert f\Vert_{L^\infty(\mathbb{R}^n)},
\end{equation*}where the constant $C>0,$ is independent of  $f\in L^\infty(\mathbb{R}^n).$ So, $\sigma(x,D):L^\infty(\mathbb{R}^n)\rightarrow BMO(\mathbb{R}^n)$ extends to a bounded operator.
\end{theorem}
\begin{proof}
Let $f\in L^\infty(\mathbb{R}^n),$ and let us denote $A=\sigma(x,D).$  Let us fix a ball $B(x_0,r)$ where $x_0\in \mathbb{R}^n,$ and $r>0.$ We will prove that there exists a constant $C>0,$ independent of $f$ and $r,$ such that
\begin{equation}\label{eq1} 
    \frac{1}{|B(x_0,r)|}\int\limits_{B(x_0,r)}|Af(x)-(Af)_{B(x_0,r)}|dx\leqslant C\Vert \sigma\Vert_{\ell,S(m^{-\frn\beta},g)}\Vert f\Vert_{L^\infty(\mathbb{R}^n)}
\end{equation}
 where 
\begin{equation*}
    (Af)_{B(x_0,r)}:=\frac{1}{|B(x_0,r)|}\int\limits_{B(x_0,r)}Af(x)dx.
\end{equation*}  
 We will split to $\sigma$ as the sum of two  symbols  
\begin{equation*}
    \sigma(x,\xi)=\sigma^0(x,\xi)+\sigma^1(x,\xi),\,\,\,\sigma^{j}(x,\xi)\in S(m^{-\frn\beta},g),\,j=0,1,
\end{equation*} in a such way that both, $\sigma^0(x,\xi)$  and  $\sigma^1(x,\xi),$ satisfy the estimate
\begin{equation}\label{symbolestimate}
    \Vert \sigma^{j}\Vert_{\ell,S(m^{-\frn\beta},g)}\leqslant C_{j,\ell}\Vert \sigma\Vert_{\ell, S(m^{-\frn\beta},g)},\,\,j=0,1,\,\ell\geqslant   1. 
\end{equation} Now, if $A^j:=\textnormal{Op}(\sigma^j),$ for $j=0,1,$ then $A=A^0+A^1$ on $C^\infty(\mathbb{R}^n),$ and we only need to prove that\begin{equation}\label{splitestimate}
    \frac{1}{|B(x_0,r)|}\int\limits_{B(x_0,r)}|A^jf(x)-(A^jf)_{B(x_0,r)}|dx\leqslant C\Vert \sigma^j\Vert_{\ell, S(m^{-\frn\beta},g)}\Vert f\Vert_{L^\infty(\mathbb{R}^n)}.
\end{equation} Then, if we prove \eqref{symbolestimate}, we can deduce \eqref{eq1}. Now, we proceed to prove the existence of $\sigma^0$ and $\sigma^1$ satisfying the requested properties.
  Let us define
 \begin{equation}\label{gamma:ytilde}
     \sigma^0(x,\xi)=\sigma(x,\xi)\tilde{\gamma}(x,\xi),\textnormal{    }\sigma^{1}(x,\xi)=\sigma(x,\xi)-\sigma^0(x,\xi),\,(x,\xi)\in \mathbb{R}^n,
 \end{equation}where $$\tilde{\gamma}(x,\xi):=\gamma(r\times m(x,\xi) )=\gamma(r\times(a_2(x,\xi)+|x|^2+\langle x, \xi\rangle)),$$ and $\gamma\in C^\infty_0(\mathbb{R},\mathbb{R}^+_0),$ is a smooth function supported in $\{t:|t|\leqslant 1\},$ satisfying $\gamma(t)=1,$ for $|t|\leqslant \frac{1}{2}.$

 In order to prove \eqref{symbolestimate} let us proceed as follows. Since $\tilde{\gamma}(x,\xi)$ is a cut-off function supported in the set $\{(x,\xi):a_2(x,\xi)+|x|^2+\langle x, \xi\rangle\leq r^{-1}\},$ in this region, the symbol $  \sigma^0(x,\xi)=\sigma(x,\xi)\tilde{\gamma}(x,\xi)$ also satisfies the $S(m,g)$-estimates in \eqref{inwhk} but probably depending on $r>0$. To remove this dependence in $r>0,$ let us use the natural Fr\'echet topology on $S(m,g)$ induced by its family of seminorms. This is the smalles topology on $S(m,g)$ which makes these seminorms continuous. Indeed, we have that
 \begin{equation*}
     \lim_{r\rightarrow 0^{+}}\Vert \sigma(x,\xi)\gamma(r\times(a_2(x,\xi)+|x|^2+\langle x, \xi\rangle))\Vert_{\ell, S(m^{-\frn\beta},g)}=\Vert \sigma(x,\xi)\Vert_{\ell, S(m^{-\frn\beta},g)},
 \end{equation*} and 
 \begin{equation*}
     \lim_{r\rightarrow \infty}\Vert \sigma(x,\xi)\gamma(r\times(a_2(x,\xi)+|x|^2+\langle x, \xi\rangle))\Vert_{\ell, S(m^{-\frn\beta},g)}=0.
 \end{equation*}In consequence, there exists $0<\varkappa<1,$ such that for all $r\in (0,\varkappa)\cup [\varkappa^{-1},\infty),$ we have that
 $$\Vert \sigma(x,\xi)\gamma(r\times(a_2(x,\xi)+|x|^2+\langle x, \xi\rangle))\Vert_{\ell, S(m^{-\frn\beta},g)}\leq 2 \Vert \sigma(x,\xi)\Vert_{\ell, S(m^{-\frn\beta},g)}.$$
 On the other hand, the function 
 \begin{equation}
     r\in [\varkappa,\varkappa^{-1}]\mapsto \sigma(x,\xi)\gamma(r\times(a_2(x,\xi)+|x|^2+\langle x, \xi\rangle))\in S(m,g),
 \end{equation}is continuous, and since $S(m,g)$ is a complete topological space, there exists $r_0\in  [\varkappa,\varkappa^{-1}] $ such that
 \begin{align*}
   & \max_{r\in [\varkappa,\varkappa^{-1}]} \Vert \sigma(x,\xi)\gamma(r\times(a_2(x,\xi)+|x|^2+\langle x, \xi\rangle))\Vert_{\ell, S(m^{-\frn\beta},g)}\\
    &=\Vert \sigma(x,\xi)\gamma(r_0\times(a_2(x,\xi)+|x|^2+\langle x, \xi\rangle))\Vert_{\ell, S(m^{-\frn\beta},g)} \lesssim_{r_0}\Vert \sigma(x,\xi)\Vert_{\ell, S(m^{-\frn\beta},g)}.
 \end{align*}Since the previous estimate depends on $r_0,$ and then on $\varkappa,$ but it is independent of $r>0,$ the estimates above prove the estimate in \eqref{symbolestimate}. 
  
 To estimate the integral
\begin{equation*}
    I_0:=\frac{1}{|B(x_0,r)|}\int\limits_{B(x_0,r)}|A^0f(x)-(A^0f)_{B(x_0,r)}|dx,
\end{equation*} we will use the mean value theorem. Indeed,  observe that
\begin{align}\label{r}
    |A^0f(x)-A^0f(y)| &\leqslant C_{0}\sum_{k=1}^{n}\Vert \partial_{x_k}A^0f\Vert_{L^\infty(\mathbb{R}^n)}|x-y|
    &\leqslant C r\sum_{k=1}^{n}\Vert \partial_{x_k}A^0f\Vert_{L^\infty(\mathbb{R}^n)}.
\end{align}
In order to estimate the $L^\infty$-norm of $\partial_{x_k}A^0f,$ first let us observe that the symbol of $\partial_{x_k}A^0=\textnormal{Op}(\sigma'_k)$ is given by
\begin{equation}\label{symbolderivative}
\sigma'_k(x,\xi):=   2\pi i \xi_k\sigma^0(x,\xi)+(\partial_{x_k}\sigma^0(x,\xi)).
\end{equation}
Indeed, the Leibniz law gives
\begin{align*}
    \partial_{x_k}A^0f(x)&=\int\limits_{\mathbb{R}^n}(\partial_{x_k}(e^{i2\pi x \xi}\sigma^0(x,\xi))\widehat{f}(\xi))\,d\xi\\
    &=\int\limits_{\mathbb{R}^2}([\partial_{x_k}(e^{i2\pi x \xi})\sigma^0(x,\xi)+e^{i2\pi x \xi}\partial_{x_k}\sigma^0(x,\xi))]\widehat{f}(\xi))\,d\xi\,,
\end{align*}
 and we obtain
\begin{align*}
     \partial_{x_k}A^0f(x)&=\int\limits_{\mathbb{R}^2}([e^{i2\pi x \xi}2\pi i \xi_k\sigma^0(x,\xi)+e^{i2\pi x \xi}\partial_{x_k}\sigma^0(x,\xi))]\widehat{f}(\xi)),
\end{align*}  which proves \eqref{symbolderivative}. By using a suitable partition of unity we will decompose $\sigma'_k(x,\xi)$ in the following way,
\begin{equation*}
    \sigma'_k(x,\xi)=\sum_{j=1}^\infty\rho_{j,k}(x,\xi).
\end{equation*}
To construct the family of operators $\rho_{j,k}(x,\xi)$ we will proceed as follows. We choose a smooth real  function $\eta$ satisfying $\eta(t)\equiv 0$ for $|t|\leqslant 1$ and $\eta(t)\equiv 1$ for $|t|\geqslant   2.$ Set
\begin{equation*}
    \rho(t)=\eta(t)-\eta(\frac{t}{2}).
\end{equation*} The support of $\rho$ satisfies $\textnormal{supp}(\rho)\subset[1,4] .$ One can check that 
\begin{equation*}
    1=\eta(t)+\sum_{j=1}^\infty\rho(2^{j}t),\,\,\,\,\textnormal{ for all }t\in \mathbb{R}.
\end{equation*} Indeed, 
\begin{equation*}
    \eta(t)+\sum_{j=1}^\ell\rho(2^{j}t)=\eta(t)+\sum_{j=1}^\ell\left(\eta(2^{j}t)-\eta(2^{j-1}t)\right)=\eta(2^{\ell}t)\rightarrow 1,\,\,\ell\rightarrow\infty.
\end{equation*}
For $t=r m(x,\xi)$ we have
\begin{equation*}
    1=\eta(rm(x,\xi))+\sum_{j=1}^\infty\rho(2^{j}rm(x,\xi)).
\end{equation*} 
 Observe that
 \begin{equation*}
     \sigma'_k(x,\xi)=\eta(rm(x,\xi))\sigma'_k(x,\xi)+\sum_{j=1}^\infty\rho(2^{j}rm(x,\xi))\sigma'_k(x,\xi).
 \end{equation*} Because $\eta(rm(x,\xi))\sigma'_k(x,\xi)=0,$ in view that the support of $\tilde\gamma$ and $\eta$ are disjoint sets, we have
  \begin{equation*}
     \sigma'_k(x,\xi)=\sum_{j=1}^\infty\rho_{j,k}(x,\xi),\,\,\,\rho_{j,k}(x,\xi):=\rho(2^{j}rm(x,\xi))\sigma'_k(x,\xi).
 \end{equation*} 
 From the mean value theorem we have,
 \begin{align*}
    |A^0f(x)-A^0f(y)| &\leqslant C_{0}\sum_{k=1}^{n}|y-x|\cdot\Vert \partial_{x_k}A^0f\Vert_{L^\infty(\mathbb{R}^n)}\\
    &\lesssim r\sum_{k=1}^{n}\Vert \partial_{x_k}A^0f\Vert_{L^\infty(\mathbb{R}^n)}.
\end{align*}
 So, we have,
 \begin{align*}
     I_0:&=\frac{1}{|B(x_0,r)|}\int\limits_{B(x_0,r)}|A^0f(x)-(A^0f)_{B(x_0,r)}|dx\\
     &= \frac{1}{|B(x_0,r)|}\int\limits_{B(x_0,r)}\left|\frac{1}{|B(x_0,r)|}\int\limits_{B(x_0,r)}(A^0f(x)-A^0f(y))dy\right|dx\\
     &\leqslant \frac{1}{|B(x_0,r)|^{2}}\int\limits_{B(x_0,r)}\int\limits_{B(x_0,r)}|A^0f(x)-A^0f(y)|dydx     \lesssim r \sup_{1\leqslant k\leqslant n}\Vert \partial_{x_k}A^0f\Vert_{L^\infty(\mathbb{R}^n)}\\
     &=r\sup_{1\leqslant k\leqslant n}\Vert \textnormal{Op}(\sigma'_{k})f\Vert_{L^\infty(\mathbb{R}^n)}\leqslant r\sup_{1\leqslant k\leqslant n}\sum_{j=1}^\infty\Vert \textnormal{Op}(\rho_{j,k})f\Vert_{L^\infty(\mathbb{R}^n)}\\
     &\lesssim r \sum_{j=1}^\infty r^{-1}2^{-j} \Vert \sigma\Vert_{\ell,S(m^{-\beta},g)}\Vert f\Vert_{L^\infty(\mathbb{R}^n)}\\
      &\lesssim rr^{-1}\Vert \sigma\Vert_{\ell,S(m^{-\frn\beta},g)}\Vert f\Vert_{L^\infty(\mathbb{R}^n)}\\
      &=\Vert \sigma\Vert_{\ell,S(m^{-\frn\beta},g)}\Vert f\Vert_{L^\infty(\mathbb{R}^n)},
\end{align*} 
where we have used that for $\ell > \frac{n}{4}$
\begin{align*}
    \Vert \textnormal{Op}(\rho_{j,k})\Vert_{\mathscr{B}(L^\infty(\mathbb{R}^n))}\lesssim \Vert \rho_{j,k}\Vert_{\ell,S(m^{-\frn\beta},g)}\lesssim r^{-1}2^{-j}\Vert \sigma\Vert_{\ell,S(m^{-\frn\beta},g)},
\end{align*}in view of Lemma \ref{lem:l} and Proposition 2.5 of \cite{Delgado2006} applied to  $\rho_{j,k}$ with $s=2^jr.$
In order to finish the proof, we need to prove \eqref{splitestimate} for $j=1.$
In order to obtain a similar $L^\infty(\mathbb{R}^n)$-${BMO}(\mathbb{R}^n)$ estimate for $A^{1},$ we will proceed as follows. Let $\phi$ be a smooth function compactly supported in $B(x_0,2r)$ satisfying that
\begin{equation*}
    \phi(x)=1, \textnormal{  for  }\,\,x\in B(x_0,r),\textnormal{ and }0\leqslant \phi\leqslant 10.
\end{equation*}
Note that
\begin{equation*}
   |B(x_0,r)|\leqslant \int\limits_{B(x_0,r)}\phi(x)^2dx\leqslant  \int\limits_{B(x_0,2r)}\phi(x)^2dx =\Vert \phi\Vert_{L^2(\mathbb{R}^n)}^2\leqslant 100  |B(x_0,2r)|,
\end{equation*} and consequently
\begin{equation}\label{doubling}
     |B(x_0,r)|^{\frac{1}{2}}\leqslant \Vert \phi\Vert_{L^2(\mathbb{R}^n)}\leqslant 10|B(x_0,2r)|^{\frac{1}{2}}\leqslant 10C|B(x_0,r)|^{\frac{1}{2}},
\end{equation}where in the last inequality we have used  that the Lebesgue measure on $\mathbb{R}^n$ satisfies the doubling property. Taking into account that
\begin{align*}
    \frac{1}{|B(x_0,r)|}\int\limits_{B(x_0,r)}|A^1f(x)-(A^1f)_{B(x_0,r)}|dx\leqslant \frac{2}{|B(x_0,r)|}\int\limits_{B(x_0,r)}|A^1f(x)|dx,
\end{align*}
we will estimate the right-hand side. Indeed, let us observe that
\begin{align*}
 &\frac{1}{|B(x_0,r)|} \int\limits_{B(x_0,r)}|A^1f(x)|dx\leqslant  \frac{1}{|B(x_0,r)|}\int\limits_{B(x_0,r)}|\phi(x)A^1f(x)|dx\\
  &\leqslant  \frac{1}{|B(x_0,r)|}\int\limits_{B(x_0,r)}|A^1[\phi f](x)|dx+\frac{1}{|B(x_0,r)|} \int\limits_{B(x_0,r)}|[M_{\phi},A^1 ] f(x)|dx\\
 :&=I+II,
\end{align*}where $M_\phi$ is the multiplication  operator by $\phi.$ To estimate $I,$ observe that, in view of the Cauchy-Schwartz inequality, we have
\begin{equation*}
    \frac{1}{|B(x_0,r)|}\int\limits_{B(x_0,r)}|A^1[\phi f](x)|dx\leqslant \frac{1}{|B(x_0,r)|^{\frac{1}{2}}}\left(\int\limits_{B(x_0,r)} |A^1[\phi f](x)|^2dx   \right)^{\frac{1}{2}}.
\end{equation*}
 Let $L:=m(x,D)^{\frn\beta\varepsilon}\in\textnormal{Op}( S(m^{\frn\beta\varepsilon},g)),$ where $0<\varepsilon<1.$
Observe that
\begin{equation}\label{L2}
   \textnormal{Op}( S(m^{-\frn\beta\varepsilon},g))\subset \textnormal{Op}( S(1,g))\subset \mathscr{B}(L^2(\mathbb{R}^n))\,.
\end{equation} Because $A^{1}\in S(m^{-\frn\beta},g), $ the Weyl-H\"ormander calculus gives
\begin{equation*}
    A^{1}L\in  \textnormal{Op}( S(m^{-\frn\beta(1-\varepsilon)},g)).
\end{equation*}
It follows from \eqref{L2} that $A^{1}L$ is bounded on $L^2(\mathbb{R}^n).$  Consequently,
\begin{align*}
 &  \frac{1}{|B(x_0,r)|^{\frac{1}{2}}}\left(\int\limits_{B(x_0,r)} |A^1L[L^{-1}(\phi f)](x)|^2dx   \right)^{\frac{1}{2}}\leqslant \frac{\Vert A^1L[L^{-1}(\phi f)] \Vert_{L^2(\mathbb{R}^n)}}{|B(x_0,r)|^{\frac{1}{2}}}\\
   &\leqslant \frac{  C\Vert \sigma_{A^1L}\Vert_{k,S(m^{-\frn\beta(1-\varepsilon)},g)}   \Vert L^{-1}( \phi f )\Vert_{L^2(\mathbb{R}^n)}}{|B(x_0,r)|^{\frac{1}{2}}}.
\end{align*}
By observing that 
\begin{equation*}
    \Vert L^{-1}(\phi f) \Vert_{L^2(\mathbb{R}^n)}=\Vert \phi f\Vert_{H(m^{-\frn\beta\varepsilon},g)},
\end{equation*}
where $H(m^{-\frn\beta\varepsilon},g)$ is the Sobolev space of order $-\frn\beta\varepsilon,$  the embedding $L^2(\mathbb{R}^n)\hookrightarrow H(m^{-\frn\beta\varepsilon},g), $ implies that
\begin{equation*}
 \Vert L^{-1}(\phi f) \Vert_{L^2(\mathbb{R}^n)}=\Vert \phi f\Vert_{H(m^{-\frn\beta\varepsilon},g)}\lesssim \Vert \phi f \Vert_{L^2(\mathbb{R}^n)} .    
\end{equation*}
Moreover, from \eqref{doubling}, we deduce the inequality
\begin{equation*}
     \Vert \phi f \Vert_{L^2(\mathbb{R}^n)}\leqslant \Vert f\Vert_{L^\infty(\mathbb{R}^n)}\Vert \phi\Vert_{L^2(\mathbb{R}^n)}\lesssim   10\Vert f\Vert_{L^\infty(\mathbb{R}^n)}|B(x_0,r)|^{\frac{1}{2}}.
\end{equation*}
So, we conclude
\begin{align*}
    I:= \frac{1}{|B(x_0,r)|}\int\limits_{B(x_0,r)}|A^1[\phi f](x)|dx&\leqslant {  C\Vert \sigma_{A^1L}\Vert_{\ell',S(1,g)} }\Vert f\Vert_{L^\infty(\mathbb{R}^n)}\\
   & \leqslant {  C\Vert \sigma_{A^1}\Vert_{\ell'',S(m^{-\frac{n\beta\varepsilon}{2}},g)} }\Vert f\Vert_{L^\infty(\mathbb{R}^n)}
\end{align*}
which is the desired estimate for $I.$ Indeed,  since  $A^1L\in S(1,g),$ we have that $A^1=A^1LL^{-1}\in S(m^{-\frac{n\beta\varepsilon}{2}},g).$ Now, we will estimate $II.$ For this, observe that the symbol of $[M_{\phi},A^1 ]=\textnormal{Op}(\theta) ,$ is given by
\begin{equation}\label{theta}
    \theta(x,\xi)=\int\limits_{\mathbb{R}^n}(\phi(x)-\phi(x-y))k(y,x-y)e^{-2\pi i y\xi}dy,
\end{equation}
where $x\mapsto k(\cdot,\cdot)\in C^\infty(\mathbb{R}^n)\otimes \mathscr{D}'(\mathbb{R}^n),$ is the integral kernel of $A^1.$ Indeed, the equality \eqref{theta}  was shown in \cite[page 364]{Delgado2006}.
Using the Taylor expansion we obtain
\begin{equation*}
    \phi(x-y)=\phi(x)+\sum_{|\alpha|=1}C_{\alpha}(\partial_{x}^{\alpha}\phi)(x)y^\alpha,
\end{equation*}
So, we can write
\begin{equation*}
    \theta(x,\xi)=\sum_{|\alpha|=1}  C_{\alpha} {\partial_{x}^{\alpha}\phi(x)}\partial_\xi^{{\alpha}} \sigma(x,\xi) .
\end{equation*}The hypothesis $\sigma \in S(m^{-\frn\beta},g), $ implies that $\partial_\xi^{{\alpha}} \sigma(x,\xi) \in  S(m^{-\frn\beta},g),$ and the fact that $\phi$ is of compact support, implies the same conclusion for $\theta:$  $\theta(x,\xi) \in  S(m^{-\frn\beta},g).$
So, one has   
\begin{align*}
   &\Vert\textnormal{Op}(\theta) \Vert_{\mathscr{B}(L^\infty(\mathbb{R}^n))}  \lesssim  \left(\sup_{(x,\xi)\in \mathbb{R}^{2n},|\alpha|\leqslant \ell'} \vert [\partial_\xi^{\alpha} \theta(x,\xi)]m(x,\xi)^{\frn\beta})  \vert     \right)\\
   &  \lesssim\sup_{|\alpha|\leqslant \ell} \Vert {\partial_{x}^{\alpha}\phi} \Vert_{L^\infty(\mathbb{R}^n)} \left(\sup_{(x,\xi)\in \mathbb{R}^{2n},|\alpha|\leqslant \ell'} \vert [\partial_\xi^{\alpha} \sigma(x,\xi)]m(x,\xi)^{\frn\beta})  \vert     \right)\\
   &\lesssim_{\ell'} \Vert \sigma\Vert_{\ell'+1,S(m^{-\frn\beta},g)}.
\end{align*}
So, we deduce
\begin{equation*}
    \frac{1}{|B(x_0,r)|} \int\limits_{B(x_0,r)}|[M_{\phi},A^1 ] f(x)|\leqslant \Vert [M_{\phi},A^1 ] f \Vert_{L^\infty(\mathbb{R}^n)},
\end{equation*} and consequently,
\begin{align*}
     \frac{1}{|B(x_0,r)|} \int\limits_{B(x_0,r)}|[M_{\phi},A^1 ] f(x)|&\leqslant\Vert\textnormal{Op}(\theta)f \Vert_{L^\infty(\mathbb{R}^n)}\\
     &\lesssim \Vert \sigma\Vert_{\ell'+1,S(m^{-\frn\beta},g)} \Vert f \Vert_{L^\infty(\mathbb{R}^n)}.
\end{align*}
Thus, we obtain 
\begin{equation*}
    II:= \frac{1}{|B(x_0,r)|}\int\limits_{B(x_0,r)}|[M_\phi,A^1 ]f(x)|dx\leqslant   C\Vert \sigma_{A}\Vert_{\ell'+1,S(m^{-\frn\beta},g)}\Vert f\Vert_{L^\infty(\mathbb{R}^n)}.
\end{equation*} 

Thus, the proof is complete.
\end{proof}

\subsection{$L^p$-boundedness}
The previous theorem and classical real interpolation  have  the following consequences on the $L^p$-boundedness for the classes $S(m^{-\frn\beta},g)$. The proof of this theorem follows standard arguments and we omit it.
\begin{theorem}\label{Lpreal} Let $\varepsilon_0=\frac{3}{4}$ and  $\varepsilon_0\leq \beta<1$ and let $\sigma(x,D)\in \textnormal{Op}S(m^{-\frn \beta},g).$ Then,  $\sigma(x,D):L^p(\mathbb{R}^n)\rightarrow L^p(\mathbb{R}^n)$ extends to a bounded operator, for all $1<p<\infty.$  Moreover, for every $\ell>\frac{n}{4},$ we have
\begin{equation*}
    \Vert \sigma(x,D)f\Vert_{L^p(\mathbb{R}^n)}\leq C\Vert \sigma(x,\xi)\Vert_{\ell,S(m^{-n\beta/2},g)}\Vert f\Vert_{{L^p(\mathbb{R}^n)}},
\end{equation*}where the constant $C>0,$ is independent of  $f\in L^p(\mathbb{R}^n),$ 
\end{theorem}

Regarding the smaller exponents $\varepsilon_0$ below, by applying the Fefferman-Stein interpolation we obtain the following bounds:
\begin{theorem}\label{Lp} Let $\varepsilon_0=\frac{3}{4}$ and  $0\leq \beta<\varepsilon_0$ and let $\sigma(x,D)\in \textnormal{Op}(S(m^{- \frn\beta},g)).$ Then, $\sigma(x,D):L^p(\mathbb{R}^n)\rightarrow L^p(\mathbb{R}^n)$ extends to a bounded operator, provided
\begin{equation}
    \left| \frac{1}{p}-\frac{1}{2}\right|\leq \frn\beta.
\end{equation}
Moreover, for such $p$ and all $\ell>\frac{n}{4}$ we have
\begin{equation*}
    \Vert \sigma(x,D)f\Vert_{L^p(\mathbb{R}^n)}\leq C\Vert \sigma(x,\xi)\Vert_{\ell,S(m^{-n\beta/2},g)}\Vert f\Vert_{{L^p(\mathbb{R}^n)}},
\end{equation*}where the constant $C>0,$ is independent of  $f\in L^p(\mathbb{R}^n).$
\end{theorem}
\begin{proof}
Now, having proved Theorem \ref{parta}, the proof is verbatim  the proof of Theorem 2.10 of \cite{Delgado2006}. We will present the argument here, for completeness of our paper. We will use the complex Fefferman-Stein interpolation theorem.  Let us consider the complex family of operators indexed by $z\in \mathbb{C},$ $ \mathfrak{Re}(z)\in [0,1],$ given by
\begin{equation*}
    T_{z}:=\textnormal{Op}(\sigma_{z}),\,\,\,\, \sigma_{z}(x,\xi):=e^{z^2}\sigma(x,\xi)m(x,\xi)^{ \frac{n\beta}{2}+{\frac{n}{2}(\frac{3}{4})(z-1)}}.
\end{equation*} The family of operators $\{T_{z}\},$ defines an analytic family of operators from $ \mathfrak{Re}(z)\in (0,1),$ (resp. continuous for $ \mathfrak{Re}(z)\in [0,1]$), into the algebra of bounded operators on $L^2(\mathbb{R}^n).$  Let us observe that $\sigma_0(x,\xi)=\sigma(x,\xi)m(x,\xi)^{{\frac{n\beta}{2}-{\frac{n}{2}(\frac{3}{4})} }},$ and $\sigma_{1}(x,\xi)=e\sigma(x,\xi)m(x,\xi)^{{\frac{n\beta}{2}}}.$ Since $T_0$ is bounded from $L^\infty(\mathbb{R}^n)$ into ${BMO}(\mathbb{R}^n)$ and $T_1$ is bounded on $L^2(\mathbb{R}^n),$ the Fefferman-Stein interpolation theorem implies that $T_t$ extends to a bounded operator on $L^p(\mathbb{R}^n),$ for $p=\frac{2}{t}$ and all $0<t\leqslant 1.$ Because $0\leqslant \beta<\frac{3}{4}, $ there exist $t_0\in (0,1)$ such that $\beta=\frac{3}{4}(1-t_{0}).$ So, $T_{t_0}=e^{t_0^2}A$ extends to a bounded operator on $L^{\frac{2}{t_0}}.$ The Fefferman-Stein interpolation theorem, the $L^2(\mathbb{R}^n)$-boundedness and the  $L^{\frac{2}{t_0}}$-boundedness  of $A$ give the $L^p(\mathbb{R}^n)$-boundedness of $A$ for all $2\leqslant p\leqslant \frac{2}{t_0},$ and interpolating the $L^{\frac{2}{t_0}}(\mathbb{R}^n)$-boundedness with the $L^\infty(\mathbb{R}^n)$-${BMO}(\mathbb{R}^n)$ boundedness of $A$ we obtain the boundedneess of $A$ on $L^p(\mathbb{R}^n)$ for all $\frac{2}{t_0}\leqslant p<\infty.$ So, $A$ extends to a bounded operator on $L^p(\mathbb{R}^n)$ for all $2\leqslant p<\infty.$ The $L^p(\mathbb{R}^n)$-boundedness of $A$ for $1<p\leqslant 2$ now follows by the duality argument.
\end{proof}
The exponent $\varepsilon_0$ in Lemma \ref{lem:l} can be improved if we dispose of more information on $a_2$. 
 For instance, if  $L$ is a second order differential operator, formally self-adjoint,
\begin{equation}\label{opl}
Lf=-\sum a_{ij}(x)\frac{\partial^2}{\partial x_{i} \partial x_{j}}f+ \mbox{ lower order terms}, \quad\quad f\in C_{0}^\infty (\mathbb{R}^n), 
\end{equation}
with the conditions as in \eqref{oprelss} and for any $x\in\mathbb{R}^n$, the matrix $a_{i,j}(x)$ has rank $r(x)\geq r_0$. 
 An important special case of the situation above arises  from the sum of squares
$$ L=-\sum\limits_{j=1}^MX_j^*X_j,$$  where $X_j$ are real vector fields on $\mathbb{R}^n$, provided that the linear span of the family of vector fields $\{X_j(x)\}_{1\leq j\leq M}$ is at least $r_0$-dimensional, $r_0\geq 1$, for each $x$. \\

We give an example on $\Ra^2$ in the case of the operator 
\begin{equation}\label{df5}a(x,D):=-(\partial_{x_{1}}^{2}+\tilde{x}_{1}^{2}\partial_{x_{2}}^{2})+|x|^{2}.\end{equation}
The exponent $\varepsilon_0=\frac{3}{4}$ can be improved by replacing it by $\frac{3}{2}$ in   Lemma \ref{lem:l}, and as a consequence improving it  as well in Theorems \ref{parta}, \ref{Lpreal} and \ref {Lp}.  \\

We now state the  improvement in Lemma \ref{lem:l} for the classes corresponding to the operator \eqref{df5}. More precisely, a detailed look at the proof of Lemma \ref{lem:l} allows one to improve the parameter $\varepsilon_0$ from $3/4$ until $5/12$ as follows.
\begin{lemma}\label{lem:lbg} Let us consider the metric $g$ given in \eqref{gen.metric}  with  $a_2(x,\xi)=\xi_{1}^{2}+\tilde{x_1}^{2}\xi_{2}^{2},$  $(x,\xi)\in\Ra^2\times\Ra^2$  being the symbol of the degenerate elliptic operator $$ a_2(x,D):=-(\partial_{x_{1}}^{2}+\tilde{x}_{1}^{2}\partial_{x_{2}}^{2}).$$ 
Let us consider the $g$-weight $m(X)=\xi_{1}^{2}+\tilde{x_1}^{2}\xi_{2}^{2}+|x|^2+\jpxe\in S(m,g).$ Let $\varepsilon_0=\frac{5}{12}$ and  $\varepsilon_0\leq\varepsilon<1.$  Consider a symbol $q\in S(m^{-\varepsilon },g)$ supported in 
    $R\leq a(x,\xi)+\jpxe\leq 3R $ for $R>1 $. Then, for all $l>1 $ we have 
\[\lvert | q(x,D)f \rvert|_{\infty}\leq C \lvert |q\rvert |_{l; S(m^{-\varepsilon },g)}  \lvert |f\rvert |_{L^\infty}, \] 
where the constant $C$ is independent of $q$ and of $f$.
\end{lemma}

\begin{proof}  
 Wee proceed like in Lemma \ref{lem:l}. 
Let $q\in S(m^{-\varepsilon },g)$ be supported in 
 \[\{(x,\xi)\in \mathbb{R}^2\times \mathbb{R}^2 \, : R\leq a(x,\xi)+\jpxe\leq 3R \}, \]
 for 
    $R>1 $ fixed.
 It is then enough to prove that for all $x\in\mathbb{R}^2$ and $ l>1,$
\[ \vert|\hat{q}_x \rvert|_{L^1}\leq C \lvert|q\rvert|_{l; S(m^{-\varepsilon },g)}.\]

\indent Let us consider  $b=R^{\varepsilon-1}.$   The  Cauchy-Schwarz inequality gives the estimate
\[\int\limits_{|y|<b}|\hat{q}_x(y)|dy\leq C\vert|q\rvert|_{l;S(m^{-\varepsilon },g)}.\]
On the other hand, for the integral $\int\limits_{|y|\geq b}|\hat{q}_x(y)|dy ,$ we first note that for an integer $k>1$,  we have that
\begin{eqnarray*}
\int\limits_{|y|\geq b}|\hat{q}_x(y)|dy \leq & Cb^{1-k }\left(\,\int\limits_{y\geq b}|y|^{2k} |\hat{q}_x(y)|^{2}dy\right)^{\half}\\
\leq & Cb^{1 -k }\left(\,\int\limits_{\mathbb{R}^2}|\nabla_{\xi}^k q(x,\xi)|^{2}d\xi\right)^{\half}\\
\leq & C\lvert|q\rvert|_{k;S(m^{-\varepsilon },g)}b^{1 -k }\left(\,\int\limits_{R\leq   a_2(x,\xi) +|x|^2+\jpxe\leq 3R}m^{-2\varepsilon -k}(x,\xi)d\xi\right)^{\half}
.\end{eqnarray*}
In order to estimate the last integral we split into two cases
:
$|x|^2\leq \frac{R}{2C}$ and $|x|^2>\frac{R}{2C}$. \\

If $
|x|^2\leq \frac{R}{2C}$, we get \[R\leq   a_2(x,\xi) +|x|^2+\jpxe\lesssim C(\xi_1^2+ |\xi_2| +|x|^2).\]
Since in the subelliptic zone $R\lesssim \xi_1^2+ |\xi_2|,$ we have that
$$ \xi_1^2+ |\xi_2|\lesssim a_2(x,\xi) +|x|^2+\jpxe=m(x,\xi), $$
we can compare
$$\int\limits_{R\leq   a_2(x,\xi) +|x|^2+\jpxe\leq 3R}m^{-2\varepsilon -k}(x,\xi)d\xi\lesssim \int\limits_{ \substack{  R\lesssim \xi_1^2+ \xi_2 \\ \xi_1,\xi_2>0}}( \xi_1^2+ |\xi_2|)^{-2\varepsilon -k}d\xi.$$ 
By making the changes of variables
$$(\xi_1,\xi_2)=(\sqrt{R}u_1, Ru_2^2),\,\,u_1,u_2>0,$$ we can estimate 
$$   \int\limits_{ \substack{  R\lesssim \xi_1^2+ \xi_2 \\ \xi_1,\xi_2>0}}( \xi_1^2+ |\xi_2|)^{-2\varepsilon -k}d\xi\lesssim\int\limits_{ \substack{  1\lesssim u_1^2+ u_2^2 \\ u_1,u_2>0}}( u_1^2+ u_2^2)^{-2\varepsilon -k}R^{-2\varepsilon -k}R^{\frac{3}{2}}u_2du_1du_2=R^{-2\varepsilon -k+\frac{3}{2}} I,$$
where the integral 
\begin{align}\label{The:integral:parametrised}
 I=\int\limits_{ \substack{  1\lesssim u_1^2+ u_2^2 \\ u_1,u_2>0}}( u_1^2+ u_2^2)^{-2\varepsilon -k}u_2du_1du_2<\infty    
\end{align}
converges since when $k>3-2\varepsilon.$ So, putting all the previous estimates together we deduce that
$$  \int\limits_{|y|\geq b}|\hat{q}_x(y)|dy\lesssim R^{(\varepsilon-1)(1-k)-\varepsilon -\frac{k}{2}+\frac{3}{4}}\leq C. $$
To see the last estimate note that $R>1$ and that
\begin{equation}
  (\varepsilon-1)(1-k)-\varepsilon -\frac{k}{2}+\frac{3}{4}=k(\frac{1}{2}-\varepsilon)-\frac{1}{4}\leq 0, 
\end{equation} if for example $k$ is the sharp parameter in the convergence of $I$ in \eqref{The:integral:parametrised}, namely, $k=3,$ from which the inequality $$ 3(\frac{1}{2}-\varepsilon)-\frac{1}{4}\leq 0,$$ from then we require that $\varepsilon\geq 5/12.$\\

On the other hand, if $|x|^2>\frac{R}{2C}$, we have
$-|x|^2<-\frac{R}{2C}$ and 
\[\xi_1^2+|\xi_2|\leq   a_2(x,\xi) +\jp\leq  3R-|x|^2\leq 3R-\frac{R}{2C}=\frac{(6C-1)}{2C}R. \]
Since $6C-1>0$ and by using the fact that $|x|^2>\frac{R}{2C}$, we have $\jpx^{-k}\leq C'R^{-\frac{k}{2}}$ and  we  obtain
$$ b^{1 -k }\left(\,\int\limits_{R\leq   a_2(x,\xi) +|x|^2+\jpxe\leq 3R}m^{-2\varepsilon -k}(x,\xi)d\xi\right)^{\half} =b^{1 -k } \left(\,\int\limits_{\xi_1^2+|\xi_2|\leq C'R}m^{-2\varepsilon}(x,\xi)m^{-k}(x,\xi)d\xi\right)^{\half}  $$

$$  \leq b^{\frn -k } \left(\int\limits_{\xi_1^2+|\xi_2|\leq C'R}(\xi_1^2+|\xi_2|)^{-2\varepsilon}\langle x\rangle^{-2k}d\xi\right)^{\half}$$
$$  \leq R^{(\varepsilon-1)(1-k) -\frac{k}{2}} \left(\int\limits_{\xi_1^2+|\xi_2|\leq C'R}(\xi_1^2+|\xi_2|)^{-2\varepsilon}d\xi\right)^{\half}$$
$$\lesssim R^{(\varepsilon-1)(1-k) -\frac{k}{2}}\left(\int\limits_{ \substack{   u_1^2+ u_2^2 \lesssim 1 \\ u_1,u_2>0}}( u_1^2+ u_2^2)^{-2\varepsilon }R^{-2\varepsilon }R^{\frac{3}{2}}u_2du_1du_2\right)^{\half}$$
$$ =R^{(\varepsilon-1)(1-k) -\frac{k}{2}}  R^{-\varepsilon +\frac{3}{4}} \sqrt{J}, $$
where 
$$ J=\int\limits_{ \substack{   u_1^2+ u_2^2 \lesssim 1 \\ u_1,u_2>0}} ( u_1^2+ u_2^2)^{-2\varepsilon }u_2du_1du_2\lesssim \int\limits_{ \substack{   u_1^2+ u_2^2 \lesssim 1 \\ u_1,u_2>0}} ( u_1^2+ u_2^2)^{-2\varepsilon }du_1du_2$$ 
is convergent for any $\varepsilon$ such that $2\varepsilon<2,$ that is when $\varepsilon<1.$ So, in the case where $|x|^2>\frac{R}{2C},$ this analysis shows that
$$  \int\limits_{|y|\geq b}|\hat{q}_x(y)|dy\lesssim R^{(\varepsilon-1)(1-k) -\frac{k}{2}}  R^{-\varepsilon +\frac{3}{4}}\leq C, $$
when e.g. $k=3,$ in view of the inequality
$$   (\varepsilon-1)(1-k) -\frac{k}{2} -\varepsilon +\frac{3}{4}=3(\frac{1}{2}-\varepsilon)-\frac{1}{4}\leq 0$$
for any $\varepsilon\geq 5/12.$
The proof of the lemma is complete.
\end{proof}
Regarding the improved exponent below $\varepsilon_0$, by applying the Fefferman-Stein interpolation we obtain the following bounds. Since its proof follows the same arguments as the ones in the proof of Theorem \ref{Lp} we omit it.
\begin{theorem}\label{Lp} Let $\varepsilon_0=\frac{5}{12}$ and  $0\leq \beta<\varepsilon_0$ and let $\sigma(x,D)\in \textnormal{Op}(S(m^{- \beta},g)).$ Then, $\sigma(x,D):L^p(\mathbb{R}^2)\rightarrow L^p(\mathbb{R}^2)$ extends to a bounded operator, provided
\begin{equation}
    \left| \frac{1}{p}-\frac{1}{2}\right|\leq \beta.
\end{equation}
Moreover, for such $p$ we have
\begin{equation*}
    \Vert \sigma(x,D)f\Vert_{L^p(\mathbb{R}^2)}\leq C\Vert \sigma(x,\xi)\Vert_{3,S(m^{-\beta},g)}\Vert f\Vert_{{L^p(\mathbb{R}^2)}},
\end{equation*}where the constant $C>0,$ is independent of  $f\in L^p(\mathbb{R}^2).$
\end{theorem}

\section{Improved $L^p$ estimates and Schatten properties for Schr\"odinger operators for H\"ormander sums of squares}\label{improvedsec}
In this section we are going to improve the estimates obtained in Lemma \ref{lem:l} for the case of sums of squares $\sum_{j=1}^kX_j^2$ of vector fields $X_j$ satisfying the H\"ormander condition of order $2$. This means that the vector fields $X_j$ and their commutators $[X_i,X_j]$ span $\mathbb{R}^{n}.$ As a consequence all the corresponding $L^p$ estimates for the classes $S(m^{-\frac{n\theta}{2}},g)$ will also be improved.  Also in this case, a sharp index $\mu_0$ that allows the membership of the operators $S(m^{-\mu},g)$ to the Schatten von Neumann classes $S_{r}(L^2(\mathbb{R}^n)),$ when $\mu>\mu_0$ is given.

\begin{remark}\label{Subelliptic:remark} For the next  properties we follow Nagel, Stein and Wainger \cite{NagelSteinWainger}, see also Jerison and S\'anchez-Calle \cite{JerisonSanchezCalle}. Let $X_1,X_2,\cdots, X_k,$ be a system of vector fields satisfying the H\"ormander condition of order 2. Then, we can endow $\mathbb{R}^n$ with a nilpotent structure as follows. If $Y_1,\cdots, Y_n,$ is an enumeration of the vectors  $X_1,X_2,\cdots, X_k,$ and of their commutators $[X_i,X_j],$ define $d_j=\textnormal{deg}(Y_j)=1,$ if $Y_j$ is one of the vector-fields $X_{i},$ and define $d_j(Y_j)=2,$ if $Y_j$ is a commutator. Let  $Q=\sum_{j}d_j$ be the homogeneous dimension associated the  Carnot-Carath\'eodory  distance $|\cdot|_{cc}$ associated with the H\"ormander  system $\{X_1,X_2,\cdots, X_k\}.$ For every $j,$ let $X_{j}(x,\xi)$ be the principal symbol of the vector-field $X_j,$ $1\leq j\leq k.$ Then:
\begin{itemize} 
\item The function
    $(x,\xi)\mapsto  
    a_2(x,\xi)=\sum_{j=1}^kX_{j}(x,\xi)^2
$ is the principal symbol of $a_{2}(x,D).$
\item For any $x\in \mathbb{R}^n,$ there exists a positive-definite matrix $(a_{i,j}(x))_{i,j=1}^n$ in such a way that
$$ a_2(x,\xi)=\sum_{j=1}^kX_{j}(x,\xi)^2=\sum_{1 \leq j \leq k\,, \\ 1 \leq \ell \leq n}a_{j,\ell}(x)\xi_j\xi_\ell,\,  $$ and with rank $r_{0}(x)\geq r_0.$
\item There exists a non-singular matrix $\theta (x),$ such that $\det(\theta(x))=1$ and with
\begin{equation}
     a_2(x,\xi)=\sum_{j=1}^{r(x)}\lambda_j(x)(\theta(x)\xi)_j,
\end{equation}where $0<\lambda_j(x),$ $1\leq j\leq r(x),$ are the positive eigenvalues of the matrix $(a_{j,\ell}(x)).$
\item The function $\xi\mapsto \overline{\xi}= \theta(x)\xi,$ and the norm
\begin{equation}
    \Vert \overline{\xi}\Vert:=\left(\sum_{j=1}^{r_0}\overline{\xi}_j^4+\sum_{j=r_0+1}\overline{\xi}_j^2\right)^{\frac{1}{4}}, 
\end{equation} is homogenous of order one at $\overline{\xi}\neq 0,$ that is, $\Vert \delta_{a}\overline{\xi}\Vert=a\Vert \overline{\xi}\Vert,$ where the family of dilations  $(\delta_a)_{a>0}$, $\delta_{a}: \Ran \rightarrow \Ran$ is defined by 
$$
(C):\quad \delta_{a}(x)=(a x_1,\cdots, ax_{r_0},a^2x_{r_0+1},\cdots, a^2 x_n).
$$ One also has the change of variable property $d(\delta_a\xi)=a^Qd\xi.$
\item On has the estimate
\begin{equation}
    (1+\Vert\overline{\xi}\Vert)\leq C m(x,\theta^{-1}(x)\xi),\,\xi\in \mathbb{R}^n, 
\end{equation} for some $C>0.$
\end{itemize}     
\end{remark}
With the notation in Remark \ref{Subelliptic:remark} we present the following lemma.
\begin{lemma}\label{Subelliptic:lemma}  Let $\mathbb{X}:=\{X_1, X_2,\dots, X_k\}$ be a family of real vector fields on $\Ran$ satisfying the H\"ormander condition of order $2.$  Let $a_2(x,D)$ be the H\"ormander sub-Laplacian
\begin{equation}\label{Hormandersum1}
    a_2(x,D)=-\sum\limits_{j=1}^kX_j^2,
\end{equation}  and let us consider the metric $g$ and the $g$-weight $m$ defined  as in \eqref{gen.metric} and  \eqref{gen.weight}, respectively. Let $\varepsilon_0=\frac{Q}{2n}$ and let  $\varepsilon_0\leq\varepsilon\leq 1.$ Consider a symbol $q\in S(m^{-\frac{n}{2}\varepsilon },g)$ supported in 
    $R\leq a(x,\xi)+\jpxe\leq 3R $ for $R>1 $. Then, for all $k\geq 1 $ we have that
\[\lvert | q(x,D)f \rvert|_{\infty}\leq C \lvert |q\rvert |_{k; S(m^{-n\varepsilon },g)}  \lvert |f\rvert |_{L^\infty}, \] 
where the constant $C$ is independent of $q$ and $f$.
\end{lemma}
\begin{proof} We only have to prove the lemma in the borderline case $\varepsilon=\varepsilon_0.$ The result for the other values of $\varepsilon\in [\varepsilon_0,1)$ follows from the inclusion properties of the $S(m,g)$-classes. Let us improve as necessary the proof given for Lemma \ref{lem:l}. Let  $q\in S(m^{-\frac{n}{2}\varepsilon },g)$ be supported in 
 \[\{(x,\xi)\in \Ran\times\Ran \, : R\leq a(x,\xi)+\jpxe\leq 3R \}, \]
 for 
    $R>1 $ fixed. In view of the convolution identity
\begin{align*}
 q(x,D)f(x)
=&\hat{q}_x*f(x),
\end{align*}
where $\hat{q}_x$ is the Fourier transform
$q_x(\xi)=q(x,\xi)$ with respect to the variable $\xi,$
 we obtain
\[|q(x,D)f(x)|\leq
    \vert|\hat{q}_x \rvert|_{L^1}  \vert|f\rvert|_{L^\infty},\,\,x\in\Ran. \]
 It is then enough to prove that for all $x\in\Ran$ and $ l>n/4 $
\[ \vert|\hat{q}_x \rvert|_{L^1}\leq C \lvert|q\rvert|_{l; S(m^{-\frac{n}{2}\varepsilon },g)}.\]

\indent Let us consider  $b>0$ to be fixed below.  We can follow the same argument as in the proof of Lemma \ref{lem:l} to deduce  by the use of the Cauchy-Schwarz inequality in order to obtain the estimate
\begin{eqnarray*}
\int\limits_{|y|<b}|\hat{q}_x(y)|dy\leq  C\vert|q\rvert|_{l;S(m^{-\frac{n}{2}\varepsilon },g)}b^{\frn}\left(\,\int\limits_{\jp\leq 3R}\jp^{-n\varepsilon}d\xi\right)^{\half}.
\end{eqnarray*}
Now, by choosing $b=R^{\varepsilon -1}$, we obtain
\[\int\limits_{|y|<b}|\hat{q}_x(y)|dy\leq C\vert|q\rvert|_{l;S(m^{-\frac{n}{2}\varepsilon },g)}.\]
On the other hand, for the integral $\int\limits_{|y|\geq b}|\hat{q}_x(y)|dy,$ we first note that for an integer $k>\frac{n}{4}$, we have $k>1$ we have
\begin{eqnarray*}
\int\limits_{|y|\geq b}|\hat{q}_x(y)|dy \leq & Cb^{\frn -k }\left(\,\int\limits_{y\geq b}|y|^{2k} |\hat{q}_x(y)|^{2}dy\right)^{\half}\\
\leq & Cb^{\frn -k }\left(\,\int\limits_{\Ran}|\nabla_{\xi}^k q(x,\xi)|^{2}d\xi\right)^{\half},
\end{eqnarray*}the Plancherel theorem implies the inequality
\begin{equation}\label{order:inequality:lemma:2:2:2}
  \int\limits_{|y|\geq b}|\hat{q}_x(y)|dy   \leq  C\lvert|q\rvert|_{k;S(m^{-n\varepsilon },g)}b^{\frn -k }\left(\,\int\limits_{R\leq   a_2(x,\xi) +|x|^2+\jpxe\leq 3R}m^{-n\varepsilon -k}(x,\xi)d\xi\right)^{\half}.
\end{equation} Let $X_{j}(x,\xi)$ be the principal symbol of the vector-field $X_j,$ $1\leq j\leq k.$ Then
\begin{equation}
    a_2(x,\xi)=\sum_{j=1}^kX_{j}(x,\xi)^2
\end{equation}is the principal symbol of $a_{2}(x,D).$ For any $x\in \mathbb{R}^n,$ in view of Remark \ref{Subelliptic:remark} we can find a positive-definite matrix $(a_{i,j}(x))_{i,j=1}^n$ in such a way that
$$ a_2(x,\xi)=\sum_{j=1}^kX_{j}(x,\xi)^2=\sum_{1 \leq j \leq k\,, \\ 1 \leq \ell \leq n}a_{j,\ell}(x)\xi_j\xi_\ell,\,  $$ and with range $r_{0}(x)\geq r_0.$ Then, there exists a non-singular matrix $\theta (x),$ such that $\det(\theta(x))=1$ and with
\begin{equation}
     a_2(x,\xi)=\sum_{j=1}^{r(x)}\lambda_j(x)(\theta(x)\xi)_j,
\end{equation}where $0<\lambda_j(x),$ $1\leq j\leq r(x),$ are the positive eigenvalues of the matrix $(a_{j,\ell}(x)).$
By defining
\begin{equation}\label{First:changes:variables}
    \overline{\xi}:=\theta(x)\xi,
\end{equation}

\begin{equation}
    \Vert \overline\xi\Vert:=\left(\sum_{j=1}^{r_0}\overline{\xi}_j^4+\sum_{j=r_0+1}^n\overline{\xi}_j^2\right)^{\frac{1}{4}}\,,
\end{equation} we  use the following family of dilations
\begin{equation}
    \delta_{r}(x):=(r \overline{\xi}_1,\cdots, r\overline{\xi}_{r_0},r^2\overline{\xi}_{r_0+1},\cdots, r^2\overline{\xi}_n), \,r>0.
\end{equation}Observe that $d(\delta_r\xi)=r^{Q}d\xi.$ The properties of $\theta(x)$ give the lower bound
\begin{equation}
    (1+\Vert\overline{\xi}\Vert)^2\lesssim  a_2(x,\theta(x)^{-1}\overline{\xi}).
\end{equation}

In order to estimate the integral in \eqref{order:inequality:lemma:2:2:2} we split into two cases
:
$|x|^2\leq \frac{R}{2C}$ and $|x|^2>\frac{R}{2C}$. \\

If $|x|^2\leq \frac{R}{2}$, let us estimate \eqref{order:inequality:lemma:2:2:2} as follows:
\begin{align*}
    \int\limits_{R\leq   a_2(x,\xi) +|x|^2+\jpxe\leq 3R}m^{-n\varepsilon -k}(x,\xi)d\xi & =\int\limits_{R\leq   a_2(x,\xi) \leq 3R-|x|^2-\jpxe}m^{-n\varepsilon -k}(x,\xi)d\xi \\
    & \leq \int\limits_{R\leq   a_2(x,\xi) }m^{-n\varepsilon -k}(x,\xi)d\xi \\
    &= \int\limits_{R\leq   a_2(x,\theta(x)^{-1}\overline{\xi})}m^{-n\varepsilon -k}(x,\theta(x)^{-1}\overline{\xi})|\det(\theta(x)^{-1})|d\overline{\xi}\\
     &\lesssim \int\limits_{\mathcal{A}}\left(\sum_{j=1}^{r_0}\overline{\xi}_j^2+\sum_{j=r_{0}+1}^{r(x)}\lambda_j(x)|\overline{\xi}_j|^2+|\overline{\xi}|\right)^{-n\varepsilon -k}d\overline{\xi},
\end{align*}
where
$$   \mathcal{A}:=\left\{\overline{\xi}:    \sum_{j=1}^{r_0}|\overline{\xi}_j|^2+\sum_{j=r_0+1}^{r(x)}\lambda_j(x)|\overline{\xi}_{j}|^2+|\overline{\xi}|\geq R\right\}.$$
The  change of variables  $\overline{\xi}:=\delta_{\sqrt{R}}\xi'$   gives the inequality
\begin{align*}
    & \int\limits_{R\leq   a_2(x,\xi) +|x|^2+\jpxe\leq 3R}m^{-n\varepsilon -k}(x,\xi)d\xi \\
     &\lesssim \sqrt{R}^{Q-2(n\varepsilon+k)}\int\limits_{\mathcal{A}'}\left(\sum_{j=1}^{r_0}\overline{\xi'}_j^2+\sum_{j=r_{0}+1}^{r(x)}R\lambda_j(x)|\overline{\xi'}_j|^2+|\overline{\xi'}|\right)^{-n\varepsilon -k}d\overline{\xi'},
\end{align*}
where 
$$   \mathcal{A}':=\left\{\overline{\xi}:    \sum_{j=1}^{r_0}|\overline{\xi}_j'|^2+\sum_{j=r_0+1}^{r(x)}R\lambda_j(x)|\overline{\xi'}_{j}|^2+|\overline{\xi'}|\geq 1\right\}.$$ 

Let us consider the set $B=\{j:\,r_{0}\leq j\leq r(x) \textnormal{ and } R\lambda_{j}(x)\geq 1\}.$ Let us apply  the change of variables $\overline{\xi}'\mapsto \overline{\xi},  $ defined by
$$\overline{\xi}_j:=\overline{\xi}_j',\,j\notin B,$$
and 
$$ \overline{\xi}_j:=\delta_{\sqrt{R\lambda_j(x)}}\overline{\xi}_j' ,\,j\in B. $$ Since $d\xi'_j=(R\lambda_{j})^{-1}d\overline{\xi}_j,$ for $j\in B,$
the Jacobian of this change is
$$ d\overline{\xi}'_1d\overline{\xi}_{2}'\cdots  d\overline{\xi}_{r_0}'\Pi_{j\in B}d\overline{\xi}_{j}'\Pi_{j\in(\{r_0+1,\cdots,r(x) \}\setminus B \cup\{r(x)+1,r(x)+2,\cdots, n\}}d\overline{\xi}_{j}' $$
$$=\Pi_{j\in B}({R\lambda_j(x)})^{-|B|} d\overline{\xi}. $$
The eigenvalue conditions $ R\lambda_{j}(x)\geq 1,$ $j\in B,$ imply that
$$|\overline{\xi}|\geq |\overline{\xi}'|\geq C \sum_{i=r(x)+1}^{n}|\overline{\xi}'_i|,\textnormal{ for some }C>0,$$
 and taking into account that $({R\lambda_j(x)})^{-|B|} \leq 1,$ for any $j\in B,$ we have the estimate
\begin{align*}
     &\int\limits_{\mathcal{A}'}\left(\sum_{j=1}^{r_0}\overline{\xi'}_j^2+\sum_{j=r_{0}+1}^{r(x)}R\lambda_j(x)|\overline{\xi'}_j|^2+|\overline{\xi'}|\right)^{-n\varepsilon -k}d\overline{\xi'}\\
     &\leq \int\limits_{\mathcal{A}'}\left(\sum_{j=1}^{r_0}\overline{\xi'}_j^2+\sum_{j=r_{0}+1}^{r(x)}R\lambda_j(x)|\overline{\xi'}_j|^2+\sum_{i=r(x)+1}^{n}|\overline{\xi}'_i|\right)^{-n\varepsilon -k}d\overline{\xi'}\\
      &\lesssim  \int\limits_{\mathcal{A}''}\left(\sum_{j=1}^{r_0}|\overline{\xi}_j|^2+\sum_{j=r_0+1}^{r(x)}|{\overline{\xi}}_{j}|^2+\sum_{i=r(x)+1}^{n}|\overline{\xi}_i|\right)^{-n\varepsilon -k}d\overline{\xi}\\
\end{align*}
where 
$$   \mathcal{A}'':=\left\{\overline{\xi}:    \sum_{j=1}^{r_0}|\overline{\xi}_j|^2+\sum_{j=r_0+1}^{r(x)}|{\overline{\xi}}_{j}|^2+|\overline{\xi}|\geq 1\right\}.$$ Then, we can compare
\begin{align*}& \int\limits_{\mathcal{A}''}\left(\sum_{j=1}^{r_0}|\overline{\xi}_j|^2+\sum_{j=r_0+1}^{r(x)}|{\overline{\xi}}_{j}|^2+\sum_{i=r(x)+1}^{n}|\overline{\xi}_i|\right)^{-n\varepsilon -k}d\overline{\xi}   \\  
     &\lesssim \int\limits_{ \{\xi':\sum_{j=1}^{r(x)}{\xi'}_j^2+|\xi'|\geq 1 \}  }\left(\sum_{j=1}^{r_0}\overline{\xi'}_j^2+|\overline{\xi}| \right)^{-n\varepsilon -k}d\overline{\xi'}\leq C(n,k,Q)<\infty
\end{align*}when
$$   n\varepsilon +k>\frac{Q}{2}.$$
So, putting all the estimates above together we have
$$   \int\limits_{|y|\geq b}|\hat{q}_x(y)|dy   \leq  C\lvert|q\rvert|_{k;S(m^{-n\varepsilon },g)}R^{(\varepsilon-1)(\frn -k) }\left(\,\int\limits_{R\leq   a_2(x,\xi) +|x|^2+\jpxe\leq 3R}m^{-n\varepsilon -k}(x,\xi)d\xi\right)^{\half}  $$
$$   \lesssim   C \Vert q\Vert_{k;  S(m^{-n\varepsilon },g)}R^{(\varepsilon-1)(\frac{n}{2} -k) } \left(\sqrt{R}^{Q-2\left(n\varepsilon+k\right)}\right)^{\frac{1}{2}} $$
$$   \lesssim   C \Vert q\Vert_{k;  S(m^{-n\varepsilon },g)}R^{(\varepsilon-1)(\frac{n}{2} -k) +\frac{1}{4}({Q-2\left(n\varepsilon+k\right))}}. $$
So, to estimate the norm $ \int\limits_{|y|\geq b}|\hat{q}_x(y)|dy$ with a bound independent of $R>1,$ we require the following two conditions
\begin{equation}\label{Req:1}
     n\varepsilon +k>\frac{Q}{2}
\end{equation} and
\begin{equation}\label{Req:2}
    (\varepsilon-1)(\frac{n}{2} -k) +\frac{1}{4}({Q-2\left(n\varepsilon+k\right))}\leq  0.
\end{equation} If $k\geq 1$ for $\varepsilon= \varepsilon_0=\frac{Q}{2n}$ we have that 
\begin{equation}
    n\varepsilon+k\geq\frac{Q}{2}+1>\frac{Q}{2} 
\end{equation} which shows that for $k\geq 1 $ the inequality in \eqref{Req:1} is automatically satisfied. Observe that
\begin{align*}
    (\varepsilon-1)\left(\frac{n}{2} -k\right) +\frac{1}{4}({Q-2\left(n\varepsilon+k\right))}&=\left(\frac{Q}{2n}-1\right)\left(\frac{n}{2}-k\right)+\frac{Q}{4}-\frac{Q}{4}-\frac{k}{2}\\
    &=\frac{Q}{4}-\frac{Qk}{2n}-\frac{n}{2}+\frac{k}{2}\leq0,
\end{align*} for any $k\geq 1.$ Indeed, since  $Q=2n-r_0,$ one has $$ Q/4-n/2=-r_0/4.$$  Note also that the inequality $Q\geq n$ implies  that $Q/2n\geq 1/2,$ from which for any $k\geq 1$ we deduce the required condition  
\begin{align*}
    \frac{Q}{4}-\frac{Qk}{2n}-\frac{n}{2}+\frac{k}{2}=\left(\frac{Q}{4}-\frac{n}{2}\right)-k\left(\frac{Q}{2n}-\frac{1}{2}\right)\leq -r_0/4\leq 0. 
\end{align*}

On the other hand, if $|x|^2>\frac{R}{2C}$, we have
$-|x|^2<-\frac{R}{2C}$ and 
\[\jp\leq   a_2(x,\xi) +\jpxe\leq  3R-|x|^2\leq 3R-\frac{R}{2C}=\frac{(6C-1)}{2C}R. \]
Since $6C-1>0$ and by using the fact that $|x|^2>\frac{R}{2C}$, we have $\jpx^{-k}\leq C'R^{-\frac{k}{2}}$ and  we  obtain
\begin{align*}
b^{\frn -k }\left(\,\int\limits_{R\leq   a_2(x,\xi) +|x|^2+\jpxe\leq 3R}m^{-n\varepsilon -k}(x,\xi)d\xi\right)^{\half}=&b^{\frn -k } \left(\,\int\limits_{\jp\leq C'R}m^{-n\varepsilon}(x,\xi)m^{-k}(x,\xi)d\xi\right)^{\half}\\
\leq &CR^{(\frac{n}{2}-k)(\varepsilon -1)}\left(\,\int\limits_{\jp\leq C'R}\jp^{-n\varepsilon}\jpx^{-2k}d\xi\right)^{\half}\\
\leq &CR^{(\frac{n}{2}-k)(\varepsilon -1)}\jpx^{-k}\left(\,\int\limits_{\jp\leq C'R}\jp^{-n\varepsilon}d\xi\right)^{\half}\\
\leq & CR^{(\frac{n}{2}-k)(\varepsilon -1)}R^{-\frac{k}{2}}R^{\frac{n}{2}(1-\varepsilon)}.
\end{align*}
By observing that 
\[R^{(\frac{n}{2}-k)(\varepsilon -1)+\frac{n}{2}(1-\varepsilon)-\frac{k}{2}}=R^{k(\frac{1}{2}-\varepsilon)}\leq C,\]
since $\varepsilon={Q}/{2n}\geq 1/2.$ The desired estimate follows.
\end{proof} 

Note that the index $\varepsilon_0=Q/2n$ is the best possible with respect to the approach that we have developed in the proof of Lemma \ref{Subelliptic:lemma}.
Since the $L^p$-boundedness of the classes $S(m,g)$ for all $1<p<\infty$ depends on the index $\varepsilon_0$ computed in Lemma \ref{Subelliptic:lemma}, the methods of this paper imply the following $L^p$-boundedness result. 
\begin{theorem}\label{Lpreal:sub} Let $\mathbb{X}:=\{X_1, X_2,\dots, X_k\}$ be a family of real vector fields on $\Ran$ satisfying the H\"ormander condition of order $2$. Let $a_2(x,D)$ be the H\"ormander sub-Laplacian
\begin{equation}\label{Hormandersum1:222}
    a_2(x,D)=-\sum\limits_{j=1}^kX_j^2,
\end{equation}  and let us consider the metric $g$ and the $g$-weight $m$ defined  as in \eqref{gen.metric} and  \eqref{gen.weight}, respectively.  Let $\varepsilon_0=\frac{Q}{2n}$ and  $\varepsilon_0\leq \beta\leq 1.$ Let $\sigma(x,D)\in \textnormal{Op}S(m^{-\frn \beta},g).$ Then,  $\sigma(x,D):L^p(\mathbb{R}^n)\rightarrow L^p(\mathbb{R}^n)$ extends to a bounded operator, for all $1<p<\infty.$  Moreover, for every $k\geq 1,$ we have
\begin{equation*}
    \Vert \sigma(x,D)f\Vert_{L^p(\mathbb{R}^n)}\leq C\Vert \sigma(x,\xi)\Vert_{k,S(m^{-n\beta/2},g)}\Vert f\Vert_{{L^p(\mathbb{R}^n)}},
\end{equation*}where the constant $C>0,$ is independent of  $f\in L^p(\mathbb{R}^n),$ 
\end{theorem}

Regarding the smaller exponents $\varepsilon_0$ below, by applying the Fefferman-Stein interpolation we obtain the following bounds:
\begin{theorem}\label{Lp:sub} Let $\mathbb{X}:=\{X_1, X_2,\dots, X_k\}$ be a family of real vector fields on $\Ran$ satisfying the H\"ormander condition of order $2$. Let $a_2(x,D)$ be the H\"ormander sub-Laplacian
\begin{equation}\label{Hormandersum1:sub}
    a_2(x,D)=-\sum\limits_{j=1}^kX_j^2,
\end{equation}  and let us consider the metric $g$ and the $g$-weight $m$ defined  as in \eqref{gen.metric} and  \eqref{gen.weight}, respectively.   Let $\varepsilon_0=\frac{Q}{2n}$ and  $0\leq \beta<\varepsilon_0.$ Let  $\sigma(x,D)\in \textnormal{Op}(S(m^{- \frn\beta},g)).$ Then, $\sigma(x,D):L^p(\mathbb{R}^n)\rightarrow L^p(\mathbb{R}^n)$ extends to a bounded operator, provided
\begin{equation}
    \left| \frac{1}{p}-\frac{1}{2}\right|\leq \frn\beta.
\end{equation}
Moreover, for such $p$ and all $\ell>\frac{n}{4}$ we have
\begin{equation*}
    \Vert \sigma(x,D)f\Vert_{L^p(\mathbb{R}^n)}\leq C\Vert \sigma(x,\xi)\Vert_{\ell,S(m^{-n\beta/2},g)}\Vert f\Vert_{{L^p(\mathbb{R}^n)}},
\end{equation*}where the constant $C>0,$ is independent of  $f\in L^p(\mathbb{R}^n).$
\end{theorem}
\begin{remark}\label{remark:general:s}
    Observe from the argument in the proof of Lemma \ref{Subelliptic:lemma} we  have the following estimate
$$  \int\limits_{R\leq   a_2(x,\xi) +|x|^2+\jpxe\leq 3R}m^{-s}(x,\xi)d\xi$$
$$ \lesssim \sqrt{R}^{Q-2s}\int\limits_{ \{\xi':\sum_{j=1}^{r(x)}{\xi'}_j^2+|\xi'|\geq 1 \}  }\left(\sum_{j=1}^{r_0}\overline{\xi'}_j^2+||\overline{\xi'}||\right)^{-s}d\overline{\xi'}<\infty,
$$  when $ s>\frac{Q}{2}$ and $R\geq 1.$ We will use this inequality to give sharp spectral properties for the $S(m,g)$ classes in the next section.
\end{remark}

\section{Schatten-von Neumann classes of the Hamiltonians}\label{Schattensec}
In this section we present some aspects of the spectral analysis of the Hamiltonians of the form \eqref{hami1}. We first discuss  some basic properties of the Schatten-von Neumann classes of operators. A preliminary introduction to the trace classes of operators can be found in \cite{Lax02}, while for the theory of Schatten-von Neumann classes we refer the reader to \cite{GK69}, \cite{RS75} or to \cite{Sim10}.

\indent Let $H$ be a separable Hilbert space over $\mathbb{C}$ endowed with an inner product $(\cdot,\cdot)$, and let $T:H \rightarrow H$ be some compact linear operator with adjoint operator $T^{*}$. Clearly $T^{*}T$ is  positive, symmetric and compact, and we can write $|T|=(T^{*}T)^{\frac{1}{2}}$ for the \textit{absolute value of $T$}.

\indent By the spectral theorem there exists an orthonormal basis for $H$ consisting of eigenfunctions of $|T|$. Let $s_n(T)$ denote the corresponding non-zero eigenvalues- also called the \textit{singular values} of the operator $T$. 

\begin{definition}
	A compact operator $T$ on a Hilbert space $H$ into $H$ belongs to the Schatten-von Neumann class $\mathit{S}_p(H)$, $1 \leq p <\infty$, if
	\begin{equation}\label{defn.sp}
	\left(\sum_{k=1}^{\infty}(s_k(T))^{p}\right)^{\frac{1}{p}}<\infty \,.
	\end{equation}
	The sum \eqref{defn.sp} is the norm of $T$ in $\mathit{S}_p$ denoted by $\parallel\cdot\parallel_{\mathit{S}_p}$. 
\end{definition}
\begin{remark}
	If we endow $\mathit{S}_{p}$, $1 \leq p <\infty$, with the norm $\parallel \cdot \parallel_{\mathit{S}_p}$,
	then $\mathit{S}_p$ becomes a Banach space.
\end{remark}
As a particular case of the spaces $\mathit{S}_p$, $1\leq p < \infty$, $\mathit{S}_2$ is the space of \textit{Hilbert-Schmidt} operators. Moreover, the Schatten-von Neumann classes are nested and we have
\[
\mathit{S}_p \subset \mathit{S}_q\,,\quad \text{if}\quad 1\leq p\leq q<\infty\,.
\]

\indent The analogy between the Schatten-von Neumann classes of operators and  functions in the Lebesgue spaces $L^r(\mathbb{R}^n)$ in the context of Weyl-H\"{o}rmander calculus, is given in the following proposition by Toft; see Remark 6.4 \cite{Tof06}.

\begin{proposition}\label{prop.toft}
We have the following:
	\begin{enumerate}[label=(\Alph*)]
		\item\label{itm:toft.shat} If $p \in [1,\infty]$, then $S(m,g) \subset s_{t,p}(\mathbb{R}^{2n})$ if and only if $m \in L^p(\mathbb{R}^{n})$.
		\item If $p \in [1,\infty]$, $a \in L^p(\mathbb{R}^{2n})\cap S(m,g)$ and $h_{g}^{N/2}m \in L^{p}(\mathbb{R}^{2n})$ for some $N \geq 0$, then $a \in s_{t,p}(\mathbb{R}^{2n})$,
	\end{enumerate}
 where we have denoted by $s_{\tau,p}(\mathbb{R}^{2n})$ the set of all symbols in $\mathcal{S}^{'}(\mathbb{R}^n)$ such that $a^{\tau}(x,D)$, $\tau \in \mathbb{R}$, belongs to $S_p(L^2(\Ran))$.
\end{proposition}
The theorem below is an application of Proposition \ref{prop.toft} to the H\"{o}rmander classes associated with the Hamiltonians of the form \eqref{hami1} that we consider here, for the special case of sum of squares.

\begin{theorem}\label{sch,clas.22}  Let $\mathbb{X}:=\{X_1, X_2,\dots, X_k\}$ be a family of real vector fields on $\Ran$ satisfying the H\"ormander condition of order $2$. Let $a_2(x,D)$ be the H\"ormander sub-Laplacian
\begin{equation}\label{Hormandersum1:sub:Schatten}
    a_2(x,D)=-\sum\limits_{j=1}^kX_j^2,
\end{equation}  and let us consider the metric $g$ and the $g$-weight $m$ defined  as in \eqref{gen.metric} and  \eqref{gen.weight}, respectively.
	Then we have
	\[
	m^{-\mu} \in L^r(\mathbb{R}^{2n})\,,
	\]
	for some $0< r < \infty$, provided that
 \begin{equation}\label{Eq:SchaSmg}
      \mu> \frac{Q}{r}=\frac{2n-r_0}{r}.
 \end{equation}
  \\
	Consequently, for such choices of $\mu$ and for $a\in S(m^{-\mu},g)$ we have
	\[
	a^{\tau}(x,D) \in \mathit{S}_r(L^2(\mathbb{R}^n))\,,\quad \text{for all}\quad \tau \in \mathbb{R}\,.
	\]
\end{theorem}
\begin{proof}
	In order to estimate the norm $\Vert m^{-\mu}\Vert_{L^r(\mathbb{R}^{2n})},$ let us take a strategy from harmonic analysis and let us make a suitable triadic decomposition of the phase space. To do this we decompose $\mathbb{R}^{2n}$ in triadic sectors
 \begin{equation}
   A_{R}:=  \{R\leq   a_2(x,\xi) +|x|^2+\jpxe\leq 3R\}
 \end{equation}where $R=3^k$ and $k\in \mathbb{Z}.$ Also, note that as $m(x,\xi)\geq 1,$ only the annulus $A_{3^{k}}$ with $k\in \mathbb{N}$ will contribute to the estimate of the norm $\Vert m^{-\mu}\Vert_{L^r(\mathbb{R}^{2n})}.$ Indeed,   by using Remark \ref{remark:general:s} with $s=\mu r/2>Q/2$ (from which we have that $\mu r>Q\geq n$) we can estimate the norm $\Vert m^{-\mu}\Vert_{L^r(\mathbb{R}^{2n})}$ as follows
\begin{align*}
   & \int_{\mathbb{R}^n}\int_{\mathbb{R}^n}m(x,\xi)^{-\mu r}\,dxd\xi\\
   &\lesssim \sum_{k=0}^\infty\int\limits_{ \{3^k\leq   a_2(x,\xi) +|x|^2+\jpxe\leq 3^{k+1}\} }m(x,\xi)^{-\mu r/2}m(x,\xi)^{-\mu r/2}dx\,d\xi\\
    &\lesssim \sum_{k=0}^\infty\int\limits_{ \{x\in \mathbb{R}^n: \langle x\rangle\leq 3^{k+1}\} } \int\limits_{\{ \xi\in \mathbb{R}^n: 3^k\leq   a_2(x,\xi) +|x|^2+\jpxe\leq 3^{k+1}\}} m(x,\xi)^{-\mu r/2}m(x,\xi)^{-\mu r/2}\,d\xi\,dx\\
     &\lesssim \sum_{k=0}^\infty\int\limits_{ \{ \langle x\rangle\leq 3^{k+1}\} } \left( |x|^2\right)^{-\mu r/2}dx\sup_{y\in \mathbb{R}^n}\int\limits_{\{3^k\leq   a_2(y,\xi) +|y|^2+\langle (y,\xi)\rangle\leq 3^{k+1}\}}m(y,\xi)^{-\mu r/2}d\xi\\
     &\lesssim \sum_{k=0}^\infty3^{-k\left(\mu r-n\right)}\int\limits_{ \{ \langle x\rangle\leq 1\} } \langle x\rangle^{-\mu r}dx \sqrt{3}^{k(Q-\mu r)}\int\limits_{ \{\xi':\sum_{j=1}^{r(y)}{\xi'}_j^2+|\xi'|\geq 1 \}  }\left(\sum_{j=1}^{r_0}\overline{\xi'}_j^2+||\overline{\xi'}||\right)^{-\mu r}d\overline{\xi'}\\
      &\lesssim \sum_{k=0}^\infty3^{-k\left(\mu r-n\right)}\sqrt{3}^{k(Q-\mu r)}= \sum_{k=0}^\infty3^{-k\left(n-\mu r\right)}3^{k({Q}/{2}-\mu r/2)}.
\end{align*}
Since
$\sum_{k=0}^\infty3^{-k\left(n-\mu r/2\right)}3^{k({Q}/{2}-\mu r/2)}\leq \sum_{k=0}^\infty3^{k/2({Q}-\mu r)}<\infty$ we have proved that the condition $\mu>Q/r $ implies that $m^{-\mu} \in L^r(\mathbb{R}^{2n})$ and then an application of condition \ref{itm:toft.shat} in Proposition \ref{prop.toft} finishes the proof of the theorem.
\end{proof}
\begin{remark}
    We note that the condition in \eqref{Eq:SchaSmg} is sharp. Indeed, by testing in the elliptic case $r_0=n$ taking the Laplacian and $r=2$, the condition \eqref{Eq:SchaSmg} coincides with the well-known sharp condition $\mu>\frac{n}{2}$ for the Hilbert-Schmidt class.
    
\end{remark}
We now derive some consequences on the rate of growth of eigenvalues for the operator $m(x,D)=a_2(x,D)+|x|^2+J_{\jpX}$, where $a_2(x,D)$ is as in Theorem \ref{sch,clas.22} and $J_{\jpX}=(-\Delta+|x|^2+1)^{\half}$. To do so, we will apply Theorem \ref{sch,clas.22} to the negative powers of $m(x,D)$,  then we get the rate of decay of eigenvalues and consequently the rate of growth.

\begin{theorem}\label{schvneig} Let $a_2(x,D)$ be a H\"ormander sub-Laplacian 
 as in Theorem \ref{sch,clas.22}. Then $m(x,D)=a_2(x,D)+|x|^2+J_{\jpX}$ satisfies
 \[m(x,D)^{-\mu}\in  \mathit{S}_r(L^2(\mathbb{R}^n))\,, \]
 provided $\mu>\frac{2n-r_0}{r}.$  \\
 
 Moreover, the sequence of eigenvalues of  $\lambda_j(a_2(x,D)+|x|^2+J_{\jpX})$ of $a_2(x,D)+|x|^2+J_{\jpX}$ has a growth of order at least
\begin{equation} j^{\frac{1}{r}}, \mbox{ as } j\rightarrow\infty,\label{fr87k}\end{equation}
provided $r>2n-r_0$. Consequently, the sequence of eigenvalues of  $\lambda_j(a_2(x,D)+|x|^2)$  has a growth at least as \eqref{fr87k} for $r>2n-r_0$.
\end{theorem}%

\begin{proof} If $\mu>\frac{2n-r_0}{r},$ the membership of $m(x,D)^{-\mu}$ into the class $\mathit{S}_r(L^2(\mathbb{R}^n))$ is an immediate  consequence of  Theorem \ref{sch,clas.22}. On the other hand, we note that from the case $\mu=1$, we can deduce the rate of decay of the eigenvalues of  $m(x,D)^{-1}$, obtaining for $r>2n-r_0$\\
\begin{equation} \label{wecpd}   
 \lambda_j(m(x,D)^{-1})=\os(j^{-\frac{1}{r}}),\, \mbox{ as }j\rightarrow\infty\,.
 \end{equation}
From this  we will  obtain the following estimate for the rate of growth of the eigenvalues of $m(x,D)=a_2(x,D)+|x|^2+J_{\jpX}:\\$

For every $L\in\ene$ there exists $L_0\in\ene$ such that
\begin{equation}
Lj^{\frac{1}{r}}\leq\lambda_j(m(x,D)),\,\, \mbox{ for }j\geq L_0.  
\end{equation}

Thus, the eigenvalues $\lambda_j(a_2(x,D)+|x|^2+J_{\jpX})$ have a growth of  order at least
\begin{equation}
 \label{EQ:growthb}
j^{\frac{1}{r}}, \mbox{ as } j\rightarrow\infty,
\end{equation}
provided $r>2n-r_0$.\\

For the last part, we observe that since $2n-r_0=n+n-r_0\geq n$ and the eigenvalues of $J_{\jpX}=(-\Delta+|x|^2+1)^{\half}$ are of order $j^{\frac{1}{r}}$ for $r>n$. It follows that the eigenvalues of  $\lambda_j(a_2(x,D)+|x|^2)$  have a growth at least as \eqref{fr87k} for $r>2n-r_0$.
\end{proof}
In order to obtain  \eqref{wecpd}  one can  apply the Weyl inequality below  which relates the singular values $s_n(T)$ and the eigenvalues $\lambda_n(T)$  for a compact operator $T$ on a complex separable Hilbert space:

\[\sum\limits_{n=1}^{\infty}|\lambda_n(T)|^p\leq \sum\limits_{n=1}^{\infty} s_n(T)^p ,\quad p>0.\]


\bibliographystyle{amsplain}

\end{document}